\documentclass[]{amsart}
\usepackage{epsfig}
\usepackage{amsmath, amssymb}

\newtheorem{thm}{Theorem}[section]
\newtheorem{lem}[thm]{Lemma}
\newtheorem{cor}[thm]{Corollary}
\newtheorem{prop}[thm]{Proposition}

\theoremstyle{remark}
\newtheorem{rmk}[thm]{Remark}

\begin{document}

\title{Thomae's Derivative Formulae for Trigonal Curves}

\author{Victor Enolski}
\author{Yaacov Kopeliovich}
\author{Shaul Zemel}


\maketitle

\section*{Introduction}

The theory of Thomae formulae started in the 19th century with \cite{[T1]} and \cite{[T2]}, where a formula generalizing the famous $\lambda$ function for hyper-elliptic curves (see \cite{[FK]} for more details about this function) was found. The formula involved a relation between theta constants evaluated at points of order 2 on the Jacobian of the hyper-elliptic curve and certain polynomials that arise naturally in the representation theory of symmetric groups. This formula should be viewed not only as the generalization of the elliptic case of genus 1, but also as a first step towards Riemann's program of finding moduli of algebraic curves through transcendental functions of period matrices.

Apart from its inherent mystery, Thomae's formula has numerous applications. Among others we mention expressing solutions of polynomial equations through theta functions, applications to cryptography (Mestre's algorithm of counting points of finite fields for hyper-elliptic curves), and deep connections to physics (especially string theory).

One can therefore ask whether Thomae's formula has a generalization to other types of curves. The most natural class of curves to investigate are the cyclic covers of $\mathbb{CP}^{1}$. The first progress was made in \cite{[BR1]} and \cite{[BR2]}, which used conformal field theory to write a formula generalizing the one of Thomae for essentially any cyclic cover. \cite{[N]} proved the formula for cyclic curves having a defining equation with a non-singular affine model (such an equation must look like $w^{n}=\prod_{i=1}^{rn}(z-\lambda_{i})$, up to branching over $\infty$). The next step was done in \cite{[EG]}, which generalized the results of \cite{[N]} (with a similar method) to curves arising from equations of the type $w^{n}=\prod_{i=1}^{k}(z-\lambda_{i})\prod_{i=1}^{k}(z-\mu_{i})^{n-1}$. Following these methods, the second author was able to derive a general formula for any fully branch cyclic cover in \cite{[Ko]}.

In another direction, a more elementary approach to theta quotients was developed in \cite{[EiF]}, and later extended in \cite{[EbF]} to the case considered in \cite{[N]}. The book \cite{[FZ]} (co-authored by the third author of this paper) gives a detailed account of this method, and extends it to several other cases. The general fully ramified cyclic case was then carried out by the third author in \cite{[Z]}. The second and third author were then able, in \cite{[KZ2]}, to state and prove a Thomae formula for a general Abelian cover of $\mathbb{CP}^{1}$, based on some machinery that was established in the prequel \cite{[KZ1]}. We also mention, in degrees 2 and 3, some explicit constructions and a proof for the formula given in \cite{[MT]}, using a variational principle from \cite{[F1]} .

Thomae's paper contains another remarkable formula, which expresses derivatives of non-singular odd theta constants of a hyper-elliptic curve $X$ of genus $g$ with a canonical homology basis $\{a_{j}\}_{j=1}^{g}$ and $\{b_{j}\}_{j=1}^{g}$ in terms of the branch points and the $a$-periods of holomorphic differentials. This \emph{Thomae derivative formula} is the topic of the present paper. Let us explain the formula for a hyper-elliptic curve $X$ whose defining function is ramified over $\infty$. Denote by $v_{s}$, $1 \leq s \leq g$ the basis for the holomorphic differentials that is dual to the $a$-part of our homology basis, set $P_{\infty}$ to be the branch point lying over $\infty$, and denote by $u_{P_{\infty}}$ the Abel--Jacobi map with base point $P_{\infty}$. Take $\mathcal{I}_{1}$ to be a set of $g-1$ branch points that does not contain $P_{\infty}$, with complement $\mathcal{J}_{1}$, and let $e(\mathcal{I}_{1})$ be the non-singular odd half-period associated with $\mathcal{I}_{1}$. Denote by $C=(C_{ls})_{l,s}$ the $g \times g$ matrix of the $a$-periods of the polynomial basis for the holomorphic differentials on $X$, and let $\Delta(\mathcal{I}_{1})$ and $\Delta(\mathcal{J}_{1})$ be the determinants of the corresponding Vandermonde matrices. We then get, for every $1 \leq s \leq g$, the formula
\[\frac{\partial}{\partial v_{s}}\theta[e(\mathcal{I}_1)]=\epsilon_{\mathcal{I}_{1}}\sqrt{\frac{\det C}{2^{g+2}\pi^{g}}}\Delta(\mathcal{I}_{1})^{1/4}\Delta(\mathcal{J}_{1})^{1/4}\sum_{l=1}^{g}C_{ls}\sigma_{g-l}(\mathcal{I}_{1}),\]
where $\sigma_{p}(\mathcal{I}_{1})$ denotes the elementary symmetric function of degree $p$ in the entries of $\mathcal{I}_{1}$, and $\epsilon_{\mathcal{I}_{1}}$ is some 8th root of unity.

Thomae's derivative formulae has not attracted much research attention, even though it has some evident applications. For example, inverting these relations yields expressions for the normalization constants of the holomorphic differentials, or for the $a$-periods of holomorphic differentials, in terms of theta constants. These formulae generalize Jacobi's relation $K=\pi\vartheta_{3}(0)^{2}/2$ for elliptic curves, mentioned in, e.g., \cite{[R]}. The study of similar relations in higher genera of hyper-elliptic curves started only recently in \cite{[ER]} and \cite{[E]}.

The formula above should also be useful in the framework of the completely integrable PDE of Korteweg-de Vries, as well as for expressions of ``winding vectors'', whose components are the above mentioned normalizing constants according to Novikov's program of the ``effectivization of the finite gap integration formulae'', that was intensely discussed during the 1970s and the 1980's (see, e.g., \cite{[D]}).

In this paper we generalize the Thomae derivative formula to the trigonal curves of the form $w^{3}=\prod_{i=1}^{3k}(z-\lambda_{i})$. As far as we are aware, this is the first generalization of such formulas since the original texts \cite{[T1]} and \cite{[T2]}. It is evident to us that these formulae have several applications and generalizations, and we plan to pursue the search for such formulae for a wider class of Riemann surfaces in subsequent publications.

\smallskip

This paper is divided into 3 sections. In Section \ref{HypEll} we illustrate our method by proving the Thomae derivative formula in the hyper-elliptic case. Section \ref{Deg3Basic} provides the basic data required for our trigonal curves, while Section \ref{ThomaeDeg3} proves the required derivative formula in this case.

\smallskip

We are thankful to the referee for a careful reading of our paper, for correcting several typos in the initial version, and for pointing out Remark \ref{q1ell} to us.

\section{Thomae Derivative formula for Hyper-Elliptic Curves \label{HypEll}}

The seminal paper \cite{[T2]} is mostly familiar due to the formulae relating branch points to the even theta constants on the genus $g$ hyper-elliptic curve $X$. However, this paper also contains a formula for non-singular odd theta derivatives. In this section we reprove this formula, in order to present the basic ideas that will be later employed to the other cases.

Let $X$ be the projective curve associated with the formula
\begin{equation}
w^{2}=f(z)=\prod_{i=1}^{2g+1}(z-\lambda_{i}),\mathrm{\ which\ is\ a\ compact\ Riemann\ surface\ of\ genus\ }g \label{curvehyperel}
\end{equation}
(and it is branched over $\infty$). Take some canonical homology basis, say $\{a_{j}\}_{j=1}^{g}$ and $\{b_{j}\}_{j=1}^{g}$, on $X$, and note that the holomorphic differentials on $X$ are naturally spanned by $\frac{z^{l-1}}{w}\mathrm{d}z$ with $1 \leq l \leq g$. On the other hand, if $v_{s}$, $1 \leq s \leq g$ is the dual basis for the holomorphic differentials (i.e., so that $\oint_{a_{j}}v_{s}=\delta_{js}$) then we define
\begin{equation}
\tau\in\mathrm{M}_{g}(\mathbb{C})\mathrm{\ having\ the\ }sj\mathrm{\ entry\ }\oint_{b_{j}}v_{s},\mathrm{\ as\ well\ as\ }C\in\mathrm{M}_{g}(\mathbb{C})\mathrm{\ with\ }C_{lj}=\oint_{a_{j}}\tfrac{z^{l-1}}{w}\mathrm{d}z. \label{matrices}
\end{equation}
The (Riemann) matrix $\tau$ lies in the Siegel upper half-space $\mathcal{H}_{g}$ of degree $g$ (since it is symmetric and its imaginary part is positive definite), and the matrix $C$ can also be seen as the transition matrix from the basis $\{v_{s}\}_{s=1}^{g}$ to the basis $\big\{\frac{z^{l-1}}{w}\mathrm{d}z\big\}_{l=1}^{g}$.

Choosing the homology basis identifies the Jacobian variety $J(X)$ of $X$ with the complex torus $\mathbb{C}^{g}/\mathbb{Z}^{g}\oplus\tau\mathbb{Z}^{g}$, via the Abel-Jacobi map $u_{P_{0}}$ with some fixed base point $P_{0}$, which we assume to be a branch point of the map $z:X\to\mathbb{CP}^{1}$.

Now, any point $\zeta\in\mathbb{C}^{g}$ can be represented as $\tau\frac{\varepsilon}{2}+I\frac{\delta}{2}$ with $\varepsilon$ and $\delta$ from $\mathbb{R}^{g}$, and for the image in $J(X)$ one simply takes the images of $\varepsilon$ and $\delta$ modulo $2\mathbb{Z}^{g}$. The vectors $\varepsilon$ and $\delta$ combine to a $2 \times g$ matrix, which is called the \emph{characteristic} $e$ of the point $\zeta$. This point is a half-period if and only if all the entries of its characteristic are integral.

\smallskip

Given any characteristic $\big[{\varepsilon \atop \delta}\big]$, one defines the associated \emph{Riemann theta function} by \[\textstyle{\theta\big[{\varepsilon \atop \delta}\big](\zeta,\tau)=\sum_{m\in\mathbb{Z}^{g}}\mathbf{e}\big[\big(m+\tfrac{\varepsilon}{2}\big)^{t}\tfrac{\tau}{2}\big(m+\tfrac{\varepsilon}{2}\big)
+\big(m+\tfrac{\varepsilon}{2}\big)^{t}\big(\zeta+\tfrac{\delta}{2}\big)\big]}\mathrm{\ for\ }\zeta\in\mathbb{C}^{g}\mathrm{\ and\ }\tau\in\mathcal{H}_{g},\] where here and throughout $\mathbf{e}(\alpha)$ stands for $e^{2\pi\imath\alpha}$ for every complex $\alpha$. It is related to the theta function $\theta=\theta\big[{0 \atop 0}\big]$ by the equality
\begin{equation}
\textstyle{\theta\big[{\varepsilon \atop \delta}\big](\zeta,\tau)=\mathbf{e}\big[\tfrac{\varepsilon^{t}\tau\varepsilon}{8}+\tfrac{\varepsilon^{t}}{2}\big(\zeta+\tfrac{\delta}{2}\big)\big]
\theta\big(\zeta+\tau\tfrac{\varepsilon^{t}}{2}+\tfrac{\delta}{2},\tau\big)}, \label{transchar}
\end{equation}
and it possesses the periodicity property
\begin{equation}
\textstyle{\theta\big[{\varepsilon \atop \delta}\big](\zeta+\tau n+l,\tau)=\mathbf{e}\big(-\tfrac{n^{t}\tau n}{2}-n^{t}\zeta+\tfrac{l^{t}\varepsilon-n^{t}\delta}{2}\big)\theta\big[{\varepsilon \atop \delta}\big](\zeta,\tau)}\mathrm{\ for\ }n\mathrm{\ and\ }l\mathrm{\ from\ }\mathbb{Z}^{g}. \label{periodicity}
\end{equation}
The additional simple properties
\[\textstyle{\theta\big[{\varepsilon \atop \delta}\big](-\zeta,\tau)=\theta\big[{-\varepsilon \atop -\delta}\big](\zeta,\tau)\quad\mathrm{and}\quad\theta\big[{\varepsilon+2n \atop \delta+2l}\big](\zeta,\tau)=\mathbf{e}\big(\tfrac{\varepsilon^{t}l}{2}\big)\theta\big[{\varepsilon \atop \delta}\big](\zeta,\tau)}\mathrm{\ when\ }n\mathrm{\ and\ }l\mathrm{\ are\ in\ }\mathbb{Z}^{g}\]
combine to show that when $\varepsilon$ and $\delta$ have integral coordinates we get
\begin{equation}
\textstyle{\theta\big[{\varepsilon \atop \delta}\big](-\zeta,\tau)=\mathbf{e}\big(-\tfrac{\varepsilon^{t}\delta}{2}\big)\theta\big[{\varepsilon \atop \delta}\big](\zeta,\tau)}. \label{halfint}
\end{equation}
Hence $\theta\big[{\varepsilon \atop \delta}\big]$ is an even function if $\varepsilon^{t}\delta$ is even, and it is an odd function when this number is odd. The corresponding characteristics are therefore called \emph{even} or \emph{odd} respectively. Among the $4^{g}$ integral characteristics there are $\frac{4^{g}+2^{g}}{2}$ even ones and $\frac{4^{g}-2^{g}}{2}$ odd ones.

The values $\theta\big[{\varepsilon \atop \delta}\big](0,\tau)$ are called \emph{theta constants}. In the integral case we say that the even characteristic $\big[{\varepsilon \atop \delta}\big]$ is \emph{non-singular} if $\theta\big[{\varepsilon \atop \delta}\big] \neq 0$. On the other hand, given an odd characteristic $\big[{\varepsilon \atop \delta}\big]$, Equation \eqref{halfint} implies that $\theta\big[{\varepsilon \atop \delta}\big]=0$, and our characteristic will be called \emph{non-singular} if $\frac{\partial}{\partial\zeta_{s}}\theta\big[{\varepsilon \atop \delta}\big](\zeta,\tau)\big|_{\zeta=0}$ does not vanish for at least for one index $1 \leq s \leq g$.

Consider the hyper-elliptic curve $X$ from Equation \eqref{curvehyperel}, let $P_{i} \in X$ be the branch point lying over $\lambda_{i}$ for $1 \leq i\leq2g+1$, and denote the branch point over $\infty$ by $P_{\infty}$. To every such $i$ we write the vector $u_{P_{\infty}}(P_{i})$ as $\tau\frac{\varepsilon_{i}}{2}+I\frac{\delta_{i}}{2}$ with integral $\varepsilon_{i}$ and $\delta_{i}$ (modulo $2\mathbb{Z}^{g}$), and denote the characteristic $\big[{\varepsilon_{i} \atop \delta_{i}}\big]$ by simply $[u_{P_{\infty}}(P_{i})]$. Note that $[u_{P_{\infty}}(P_{\infty})]$ is defined similarly, and yields just $\big[{0 \atop 0}\big]$.

We shall make use of the following result, for a proof of which one may consult \cite{[FK]}.
\begin{prop}
The homology basis $\{a_{j}\}_{j=1}^{g}$ and $\{b_{j}\}_{j=1}^{g}$ is completely determined by the characteristics $[u_{P_{\infty}}(P_{i})]$ with $1 \leq i\leq2g+1$. Of these characteristics $g$ are odd and $g+1$ are even, and $[u_{P_{\infty}}(P_{\infty})]$ is even as well. If $\mathcal{I}^{odd}$ denotes the set of those indices $i$ yielding odd characteristics, then the vector $K_{P_{\infty}}$ of Riemann constants that is associated with $P_{\infty}$ is given by \[K_{P_{\infty}}=\sum_{i\in\mathcal{I}^{odd}}u_{P_{\infty}}(P_{i})=-\sum_{i\in\mathcal{I}^{odd}}u_{P_{\infty}}(P_{i})\quad(\mathrm{it\ is\ of\ order\ }2).\] \label{Kvector}
\end{prop}

In the hyper-elliptic case, the characteristics that are relevant for the Thomae formulae (involving either the constants or the derivatives) can be given in terms of certain partitions of the branch points. Recall that for a real number $t$ the symbol $\lfloor t \rfloor$ stands for the largest integer that does not exceed $t$, and that $|Y|$ denotes the cardinality of the finite set $Y$. Then there is one-to-one correspondence, given in, e.g., page 13 of \cite{[F1]}, between the integral characteristics $\big[{\varepsilon \atop \delta}\big]$ and the union over $0 \leq m\leq\big\lfloor\frac{g+1}{2}\big\rfloor$ of partitions of the set of indices $1 \leq i \leq2g+1$ together with $\infty$ into the disjoint union of two sets
\begin{equation}
\mathcal{I}_{m}\cup\mathcal{J}_{m},\mathrm{\ with\ }|\mathcal{I}_{m}|=g+1-2m\mathrm{\ and\ }|\mathcal{J}_{m}|=g+1+2m,\mathrm{\ where\ for\ }m=0\mathrm{\ we\ have\ }\infty\in\mathcal{I}_{0}. \label{partition}
\end{equation}
Given such $m$, there are $\binom{2g+2}{g+1-2m}$ possible choices in Equation \eqref{partition} if $m\geq1$ and $\binom{2g+1}{g}$ choices when $m=0$ (in total indeed $4^{g}$ characteristics), and given such a partition the corresponding characteristic arises from the vector
\begin{equation}
e(\mathcal{I}_{m})=\sum_{i\in\mathcal{I}_{m}}u_{P_{\infty}}(P_{i})+K_{P_{\infty}}. \label{charIm}
\end{equation}
In this case $m$ is the index of speciality of the divisor arising from Equation \eqref{charIm} (in particular, the parity of $[e(\mathcal{I}_{m})]$ is the parity of $m$). We will be interested in the case $m=0$ of non-singular even theta constants, as well as the case $m=1$ of non-singular odd theta derivatives.

Given any subset $\mathcal{I}$ of $\{i\in\mathbb{N}|1 \leq i\leq2g+1\}\cup\{\infty\}$, we define the expression
\begin{equation}
\Delta(\mathcal{I})=\prod_{i,j\in\mathcal{I}\setminus\{\infty\},\ i<j}(\lambda_{i}-\lambda_{j}), \label{deltas}
\end{equation}
which can be viewed as a determinant of a Vandermonde matrix. Then the usual Thomae formula, involving theta constants, reads as follows.
\begin{thm}
Let $\mathcal{I}_{0}\cup\mathcal{J}_{0}$ be as in Equation \eqref{partition} with $m=0$, and set $[e(\mathcal{I}_{0})]$ to be the corresponding characteristic from Equation \eqref{charIm} and $C$ as in Equation \eqref{matrices}. Then we have
\[\theta[e(\mathcal{I}_{0})](0,\tau)=\epsilon_{\mathcal{I}_{0}}\bigg(\frac{\det C}{2^{g}\pi^{g}}\bigg)^{1/2}\Delta(\mathcal{I}_{0})^{1/4}\Delta(\mathcal{J}_{0})^{1/4},\] where $\Delta(\mathcal{I}_{0})$ and $\Delta(\mathcal{J}_{0})$ are defined in Equation \eqref{deltas}, and $\epsilon_{\mathcal{I}_{0}}$ is an 8th root of unity. \label{Thomae1Z2}
\end{thm}
We can also state the Thomae derivative formula.
\begin{thm}
Let $\mathcal{I}_{1}\cup\mathcal{J}_{1}$ be a partition as in Equation \eqref{partition} with $m=1$ and with $\infty\not\in\mathcal{I}_{1}$, and consider the Vandermonde determinants $\Delta(\mathcal{I}_{1})$ and $\Delta(\mathcal{J}_{1})$ from Equation \eqref{deltas}. Then for every $1 \leq s \leq g$ we have the equality \[\frac{\partial}{\partial\zeta_{s}}\theta[e(\mathcal{I}_{1})](\zeta,\tau)\bigg|_{\zeta=0}=
\epsilon_{\mathcal{I}_{1}}\bigg(\frac{\det C}{2^{g+2}\pi^{g}}\bigg)^{1/2}\Delta(\mathcal{I}_{1})^{1/4}\Delta(\mathcal{J}_{1})^{1/4}
\sum_{l=1}^{g}C_{ls}(-1)^{g-l}\sigma_{g-l}(\mathcal{I}_{1}),\] where $C$ and $C_{ls}$ are defined in Equation \eqref{matrices}, $\epsilon_{\mathcal{I}_{1}}$ is a 8th root of unity that is independent of $s$, and $\sigma_{p}(\mathcal{I}_{1})$ is the elementary symmetric function of degree $p$ in $\{\lambda_{i}\}_{i\in\mathcal{I}_{1}}$. \label{Thomae2Z2}
\end{thm}
The proof of Theorem \ref{Thomae1Z2} is well-documented, see, e.g., \cite{[F1]}. We now present an elementary proof of Theorem \ref{Thomae2Z2}. We recall the notation $f(z)$ from Equation \eqref{curvehyperel}, and prove the following lemma.
\begin{lem}
For any $1 \leq k\leq2g+1$ there exists a 4th root of unity $\tilde{\epsilon}_{k}$ such that as functions of the $g$ points $Q_{r}$, $1 \leq r \leq g$ on $X$ we have the equality \[\frac{\theta[u_{P_{\infty}}(P_{k})]^{2}\big(\sum_{r=1}^{g}u_{P_{\infty}}(Q_{r})+K_{P_{\infty}},\tau\big)}
{\theta^{2}\big(\sum_{r=1}^{g}u_{P_{\infty}}(Q_{r})+K_{P_{\infty}},\tau\big)}=\frac{\tilde{\epsilon}_{k}}{\sqrt{f'(\lambda_{k})}}\prod_{r=1}^{g}\big(\lambda_{k}-z(Q_{r})\big).\] \label{thetaquotient}
\end{lem}

\begin{proof}
Consider the expression
\begin{equation}
F(Q_{1},\ldots,Q_{g})=\frac{\theta^{2}\big(u_{P_{\infty}}(P_{k})+\sum_{r=1}^{g}u_{P_{\infty}}(Q_{r})+K_{P_{\infty}}\big)}
{\theta^{2}\big(\sum_{r=1}^{g}u_{P_{\infty}}(Q_{r})+K_{P_{\infty}}\big)} \label{quotdef}
\end{equation}
at a neighborhood in $X^{g}$ in which all the points are distinct from one another and from $P_{\infty}$, and such that the divisor $\sum_{r=1}^{g}Q_{r}$ is non-special (so that the denominator does not vanish). Fix an index $1 \leq i \leq g$, and then there is only one positive canonical divisor that contains all the points $Q_{r}$ with $r \neq i$ in its support (since $\sum_{r \neq i}Q_{r}$ has index of specialty 1 by Riemann--Roch). Write this canonical divisor as $\sum_{r \neq i}Q_{r}+\sum_{t=1}^{g-1}R_{t}^{(i)}$, and consider the expression from Equation \eqref{quotdef} as a function of $Q_{i}$, with the other $Q_{r}$s fixed. By the Riemann Vanishing Theorem, the numerator has zeros of order 2 at $P_{k}$ and the points $R_{t}^{(i)}$ with $1 \leq t \leq g-1$ (counted with multiplicities), while the zeros of the denominator are $P_{\infty}$ and the $R_{t}^{(i)}$'s, again of order 2. By Equation \eqref{transchar} this also describes the zeros and poles of the desired quotient (again as a function of $Q_{i}$ alone), and since Equation \eqref{periodicity} shows that this expression does not change under adding periods, we find that it is a meromorphic function on $X$. Our evaluation of its divisor determines this function as a constant multiple of $\lambda_{k}-z(Q_{i})$.

Now, since $i$ was arbitrary, we find that the only dependence of the desired left hand side on every $Q_{r}$ is via a multiplier of $\lambda_{k}-z(Q_{r})$. This yields the equality
\[\frac{\theta[u_{P_{\infty}}(P_{k})]^{2}\big(\sum_{r=1}^{g}u_{P_{\infty}}(Q_{r})+K_{P_{\infty}}\big)}{\theta^{2}\big(\sum_{r=1}^{g}u_{P_{\infty}}(Q_{r})
+K_{P_{\infty}}\big)}=c\prod_{r=1}^{g}\big(\lambda_{k}-z(Q_{r})\big)\] for some constant $c$ that is independent of the $Q_{r}$s. To determine the value of $c$ we apply Equation \eqref{transchar} once to the denominator and twice to the numerator, and find that if \[\sum_{r=1}^{g}u_{P_{\infty}}(Q_{r})+K_{P_{\infty}}=\tau\tfrac{\varepsilon}{2}+I\tfrac{\delta}{2}\quad\mathrm{and}\quad u_{P_{\infty}}(P_{k})=\tau\tfrac{\mu}{2}+I\tfrac{\nu}{2}\quad\mathrm{modulo}\quad\tau\mathbb{Z}^{g}\oplus\mathbb{Z}^{g},\] then the quotient on the left hand side is \[\mathbf{e}\big(-\tfrac{\varepsilon^{t}\nu}{2}\big)\frac{\theta\big[u_{P_{\infty}}(P_{k})+\sum_{r=1}^{g}u_{P_{\infty}}(Q_{r})+K_{P_{\infty}}\big]^{2}
(0,\tau)}{\theta\big[\sum_{r=1}^{g}u_{P_{\infty}}(Q_{r})+K_{P_{\infty}}\big]^{2}(0,\tau)}.\] Take now the $Q_{r}$s to be the elements of a set $\mathcal{I}_{0}$ as in Equation \eqref{partition} that does not contain $P_{k}$, so that $\varepsilon$ has integral entries and the characteristic in the latter denominator is just $[e(\mathcal{I}_{0})]$. Then adding $u_{P_{\infty}}(P_{k})$ replaces $\infty\in\mathcal{I}_{0}$ by $k$, and we denote this set by $\tilde{\mathcal{J}}_{0}$ and its complement by $\tilde{\mathcal{I}}_{0}$ (since Equation \eqref{partition} requires $\infty$ to be in the part denoted with $\mathcal{I}$). It follows that the numerator involves the characteristic $[e(\tilde{\mathcal{I}}_{0})]$, and by applying Theorem \ref{Thomae1Z2} to both the numerator and denominator we deduce that \[c\mathrm{\ is\ a\ 4th\ root\ of\ unity\ times\ }\frac{\Delta(\tilde{\mathcal{I}}_{0})^{1/2}\Delta(\tilde{\mathcal{J}}_{0})^{1/2}}{\Delta(\mathcal{I}_{0})^{1/2}\Delta(\mathcal{J}_{0})^{1/2}\prod_{i\in\mathcal{I}_{0}\setminus\{\infty\}}(\lambda_{k}-\lambda_{i})}\] (the product in the denominator is the product of the expressions $\lambda_{k}-z(Q_{r})$ with our choice of $Q_{r}$, $1 \leq r \leq g$). Since the relations between $\mathcal{I}_{0}$ and $\tilde{\mathcal{J}}_{0}$ and between $\mathcal{J}_{0}$ and $\tilde{\mathcal{I}}_{0}$ reduce the latter quotient (up to an additional 4th root of unity) to \[\frac{1}{\prod_{1 \leq i\leq2g+1,i \neq k}(\lambda_{k}-\lambda_{i})^{1/2}}=\frac{1}{\sqrt{f'(\lambda_{k})}}\] (by the standard evaluation of the latter derivative), this proves the lemma.
\end{proof}

In order to prove Theorem \ref{Thomae2Z2} we shall differentiate the expression from Lemma \ref{thetaquotient} with respect to $\zeta_{s}$. For this we need to know the derivatives $\frac{\partial z(Q_{r})}{\partial\zeta_{s}}$.
\begin{lem}
Consider the function $\zeta=\sum_{r=1}^{g}u_{P_{\infty}}(Q_{r})$ of the $g$ points $\{Q_{r}\}_{r=1}^{g}$, and assume that the values $\{z(Q_{r})\}_{r=1}^{g}$ are distinct and that no $Q_{r}$ is a branch point of $z$. Let $C$ be as in Equation \eqref{matrices}, and set $\mathcal{A}$ to be the matrix with $lr$-entry $\frac{z(Q_{r})^{l-1}}{w(Q_{r})}$. Then we have the equality \[\frac{\partial z(Q_{r})}{\partial\zeta_{s}}=(\mathcal{A}^{-1}C)_{rs}=\frac{w(Q_{r})}{\prod_{i \neq r}\big(z(Q_{r})-z(Q_{i})\big)}\sum_{l=1}^{g}(-1)^{g-l}\sigma_{g-l}^{(r)}(Q_{1},\ldots,Q_{g})C_{ls}.\] \label{dQrdzetas}
\end{lem}

\begin{proof}
Define the function
\begin{equation}
F(z)=\prod_{r=1}^{g}\big(z-z(Q_{r})\big),\quad\mathrm{as\ well\ as}\quad F_{r}(z)=\tfrac{F(z)}{z-z(Q_{r})}=\sum_{p=0}^{g-1}(-1)^{p}\sigma_{p}^{(r)}(Q_{1},\ldots,Q_{g})z^{g-1-p} \label{polprod}
\end{equation}
for $1 \leq r \leq g$, where $\sigma_{p}^{(r)}(Q_{1},\ldots,Q_{g})$ is the elementary symmetric functions of order $p$ evaluated at the elements $\{z(Q_{i})\}_{i \neq r}$. Recall that the Abel--Jacobi images appearing in the formula for $\zeta$ are based on the integration of the dual basis $\{v_{s}\}_{s=1}^{g}$, and that the latter basis is taken to the basis $\big\{\frac{z^{l-1}}{w}\mathrm{d}z\big\}$ by the invertible matrix $C$. It follows that $\sum_{s=1}^{g}C_{ls}\zeta_{s}$ is a constant plus $\sum_{r=1}^{g}\int_{P_{\infty}}^{Q_{r}}\frac{z^{l-1}}{w}\mathrm{d}z$, and since $C$ does not depend on $z$, the derivative of the latter expression with respect to $z(Q_{r})$ is simply $\frac{z(Q_{r})^{l-1}}{w(Q_{r})}$. Multiplication by $C^{-1}$ implies that the matrix of derivatives of $\zeta=\sum_{r=1}^{g}u_{P_{\infty}}(Q_{r})+K_{P_{\infty}}$ with respect to the parameters $z(Q_{r})$s is $C^{-1}\mathcal{A}$. The description of the required derivatives in terms of the inverse matrix $\mathcal{A}^{-1}C$ is thus established, and since $\mathcal{A}$ is the product of a Vandermonde matrix and a diagonal matrix, we have an explicit formula for its inverse: The numerators $w(Q_{r})$ come from inverting the diagonal matrix, and in the $r$th row of the inverse of the Vandermonde matrix we have the entries of the polynomial of degree $g-1$ that vanishes at the points $z(Q_{i})$ for $i \neq r$ and sends $z(Q_{r})$ to 1. As this polynomials must be a multiple of $F_{r}$ from Equation \eqref{polprod}, and the multiplier must be the inverse of $F_{r}\big(z(Q_{r})\big)=F'\big(z(Q_{r})\big)$ (which is the asserted product), this yields the desired explicit expression. This proves the lemma.
\end{proof}

In view of what happens in the $Z_{3}$ curves case below, let us indicate the form of Lemma \ref{dQrdzetas} when each $Q_{r}$ lies near a branch point $P_{i}$ (and these branch points are distinct). Then we have the natural coordinate $t_{r}(Q_{r})$ in that neighborhood, which is defined by the equality
\begin{equation}
t_{r}^{2}=z(Q_{r})-\lambda_{i}\quad\mathrm{and\ also\ satisfies}\quad w_{r}=t_{r}\big(\sqrt{f'(\lambda_{i})}+O(t_{r}^{2})\big). \label{trzr2}
\end{equation}
Considering the derivatives of these coordinates in Lemma \ref{dQrdzetas}, a similar argument to the proof will multiply each column of $\mathcal{A}$ (or equivalently of $C^{-1}\mathcal{A}$) by $\frac{\partial z(Q_{r})}{\partial t_{r}}=2t_{r}$, which replaces the expression $\frac{z(Q_{r})^{l-1}}{w(Q_{r})}$ by $\frac{\lambda_{i}^{l-1}}{2\sqrt{f'(\lambda_{i})}}+O(t_{r}^{2})$. The branch point $P_{i}$ itself corresponds to $t_{r}=0$, which yields again a Vandermonde matrix, with the value at the $r$th column being $\lambda_{i}$ and the column itself being multiplied by $\frac{2}{\sqrt{f'(\lambda_{i})}}$. Hence $\frac{\partial t_{r}}{\partial\zeta_{s}}$ is given (when $\zeta$ is the image of the branch points themselves) by the expression from Lemma \ref{dQrdzetas}, but with the coefficient replaced by the non-zero multiplier $\frac{\sqrt{f'(\lambda_{i})}}{2F'(\lambda_{i})}$.

We can now prove the Thomae derivative formula for the hyper-elliptic case.
\begin{proof}[Proof of Theorem \ref{Thomae2Z2}]
We differentiate the quotient from Lemma \ref{thetaquotient} without the square. Note that the right hand side there is $\frac{\tilde{\epsilon}_{k}}{\sqrt{f'(\lambda_{k})}}F(\lambda_{k})$ in the notation of Equation \eqref{polprod}, and its derivative with respect to $z(Q_{r})$ is $\frac{-\tilde{\epsilon}_{k}}{\sqrt{f'(\lambda_{k})}}F_{r}(\lambda_{k})$. Lemma \ref{dQrdzetas} thus yields
\begin{equation}
\frac{\partial}{\partial\zeta_{s}}\frac{\theta[u_{P_{\infty}}(P_{k})]\big(\sum_{r=1}^{g}u_{P_{\infty}}(Q_{r})+K_{P_{\infty}},\tau\big)}{\theta\big(\sum_{r=1}^{g}u_{P_{\infty}}(Q_{r})+K_{P_{\infty}},\tau\big)}=
\frac{-\sqrt{\tilde{\epsilon}_{k}}}{2f'(\lambda_{k})^{1/4}\sqrt{F(\lambda_{k})}}\sum_{r=1}^{s}(\mathcal{A}^{-1}C)_{rs}F_{r}(\lambda_{k}). \label{derivative}
\end{equation}
Equation \eqref{derivative} clearly extends by continuity to the case where $z$ is not the local coordinate at some of the points $Q_{r}$ (i.e., where we use the coordinates $t_{r}(Q_{r})$ from Equation \eqref{trzr2}), as long as they remain distinct from $P_{\infty}$, from $P_{k}$, and from one another. In this case the coefficients $w_{r}$ from the expression for $(\mathcal{A}^{-1}C)_{rs}$ in Lemma \ref{dQrdzetas} will simply vanish for these $r$s. Considering the behavior of $\sqrt{F(\lambda_{k})}$, $w(Q_{m})$, and $F_{r}(\lambda_{k})$ for all $r$ as $Q_{m} \to P_{k}$ for some $1 \leq m \leq k$, one sees that the terms with $r \neq m$ vanish at this limit, and for the remaining term we get
\begin{equation}
\lim_{Q_{m} \to P_{k}}\frac{-\sqrt{\tilde{\epsilon}_{k}}}{2f'(\lambda_{k})^{1/4}\sqrt{F(\lambda_{k})}}\frac{w(Q_{m})}{F'\big(z(Q_{m})\big)}F_{m}(\lambda_{k})=\frac{\hat{\epsilon}_{k}f'(\lambda_{k})^{1/4}}{2\sqrt{F'(\lambda_{k})}} \label{QmtoPk}
\end{equation}
(all multiplied by $\sum_{l=1}^{g}(-1)^{g-l}\sigma_{g-l}^{(r)}(Q_{1},\ldots,Q_{g})C_{ls}$), for some 8th root of unity $\hat{\epsilon}_{k}$.

Take now our partition $\mathcal{I}_{1}\cup\mathcal{J}_{1}$ (with $\infty\not\in\mathcal{I}_{1}$), and assume that $k\not\in\mathcal{I}_{1}$ either. Then Equation \eqref{partition} shows that setting $\mathcal{I}_{0}=\mathcal{I}_{1}\cup\{k,\infty\}$ and $\mathcal{J}_{0}=\mathcal{J}_{1}\setminus\{k,\infty\}$ yields an admissible partition. We set the $Q_{r}$s to be the $P_{i}$s with $i\in\mathcal{I}_{1}\cup\{k\}=\mathcal{I}_{0}\setminus\{\infty\}$ (in some order), so that the roots of the functions from Equation \eqref{polprod} are at the points $\lambda_{i}$ with $i$ in that set. In this case Equation \eqref{transchar} relates the theta function from the numerator near that point to $\theta[e(\mathcal{I}_{1})](\zeta,\tau)$ near $\zeta=0$, and since the corresponding theta constant vanishes, one has to differentiate neither the denominator nor the exponential function from Equation \eqref{transchar} for this value. As we have $Q_{m}=P_{k}$ for some $m$, we apply Equation \eqref{QmtoPk}, and the expressions $\sigma_{g-l}^{(m)}(Q_{1},\ldots,Q_{g})$ for $1 \leq l \leq g$ with this $m$ are $\sigma_{g-l}(\mathcal{I}_{1})$ by definition. Observing that $F(z)$ is $\prod_{i\in\mathcal{I}_{1}\cup\{k,\infty\}}(z-\lambda_{i})$, and that the exponents from Equation \eqref{transchar} that relate the numerator to $\theta[e(\mathcal{I}_{1})](\zeta,\tau)$ near $\zeta=0$ and the denominator to $\theta[e(\mathcal{I}_{0})](0,\tau)$ cancel to a 4th root of unity $\mu$, Equations \eqref{derivative} and \eqref{QmtoPk} reduce to the equality
\begin{equation}
\frac{\partial\theta[e(\mathcal{I}_{1})](\zeta,\tau)}{\partial\zeta_{s}}=\mu\frac{\hat{\epsilon}_{k}\prod_{i\in\mathcal{J}_{0}}(\lambda_{k}-\lambda_{i})^{1/4}}{2\prod_{i\in\mathcal{I}_{1}}(\lambda_{k}-\lambda_{i})^{1/4}}
\cdot\theta[e(\mathcal{I}_{0})](0,\tau)\cdot\sum_{l=1}^{g}(-1)^{g-l}\sigma_{g-l}(\mathcal{I}_{1})C_{ls}. \label{tosub}
\end{equation}
We evaluate $\theta[e(\mathcal{I}_{0})](0,\tau)$ using Theorem \ref{Thomae1Z2}, and note that the product in the numerator of Equation \eqref{tosub} is $\frac{\Delta(\mathcal{J}_{1})^{1/4}}{\Delta(\mathcal{J}_{0})^{1/4}}$, and the one in the denominator there is $\frac{\Delta(\mathcal{I}_{0})^{1/4}}{\Delta(\mathcal{I}_{1})^{1/4}}$ (both up to 4th roots of unity). Since the extra factor of 2 from the denominator in Equation \eqref{tosub} combines with the external coefficient from Theorem \ref{Thomae1Z2} to the asserted one, we may gather the 8th roots of unity, whose construction did not depend on $s$, to $\epsilon_{\mathcal{I}_{1}}$, which completes the proof of the theorem.
\end{proof}

The extension of Equation \eqref{derivative} to the case where a point $Q_{r}$ tends to a branch point $P_{i}$ on $X$ can also be explained in terms of the coordinate $t_{r}(Q_{r})$, where differentiating $\lambda_{k}-z(Q_{r})$ yields also a multiplier of $t_{r}$, which vanishes at $t_{r}=0$ (as we have seen in the proof above). On the other hand, when $Q_{m} \to P_{k}$ we have a multiplier of $t_{m}(Q_{m})$ in $\sqrt{F(\lambda_{k})}$ in the denominator, which cancels with this factor of $t_{m}$ and indeed yields at $t_{m}=0$ the finite, non-zero value from Equation \eqref{QmtoPk}.

Considering the $g$ subsets $\mathcal{I}_{1}^{(k)}$ of cardinality $g-1$ inside a set $\mathcal{I}_{0}\setminus\{\infty\}$ for $\mathcal{I}_{0}$ as in Equation \eqref{partition} (with the respective complements $\mathcal{J}_{1}^{(k)}$), we may write Theorem \ref{Thomae2Z2} in matrix form as
\begin{equation}
\frac{\partial\big(\theta[e(\mathcal{I}_{1}^{(1)})],\ldots,\theta[e(\mathcal{I}_{1}^{(g)})]\big)}{\partial(\zeta_{1},\ldots,\zeta_{g})}
\bigg|_{\zeta=0}=\bigg(\frac{\det C}{2^{g+2}\pi^{g}}\bigg)^{1/2}\mathcal{D}\Sigma C \label{matrixthomae2}
\end{equation}
(i.e., $\frac{\partial}{\partial\zeta_{s}}\theta[e(\mathcal{I}_{1}^{(k)})](\zeta,\tau)\big|_{\zeta=0}$ is $\big(\frac{\det C}{2^{g+2}\pi^{g}}\big)^{1/2}$ times the $ks$-entry of $\mathcal{D}\Sigma C$), where \[\Sigma_{kl}=(-1)^{g-l}\sigma_{g-l}(\mathcal{I}_{1}^{(k)})=(-1)^{g-l}\sigma_{g-l}^{(k)}(\mathcal{I}_{0})\quad\mathrm{and}\quad
\mathcal{D}_{ki}=\delta_{ki}\epsilon_{\mathcal{I}_{1}^{(k)}}\Delta(\mathcal{I}_{1}^{(k)})^{1/4}\Delta(\mathcal{J}_{1}^{(k)})^{1/4}\] (i.e.,
$\mathcal{D}$ is an invertible diagonal matrix, and $\Sigma$ is also invertible, with determinant $\Delta(I_{0})$, as an inverse Vandermonde matrix with the rows multiplied by appropriate scalars). Equation \eqref{matrixthomae2} resolves the ``effectivization'' problem for hyper-elliptic curves, mentioned in the Introduction: The columns of the normalizing matrix $C^{-1}$ for the holomorphic differentials are expressible in terms of theta constants. In addition, taking the determinant of the matrix from Equation \eqref{matrixthomae2} one can derive the \emph{Riemann--Jacobi derivative formula} for the hyper-elliptic curve (see, e.g., \cite{[F2]}). These applications are examples of the importance of the Thomae derivative formula, and can be seen as motivations for generalizing this formula for other cases, like the case of non-singular trigonal curves considered in what follows.

Theorem \ref{Thomae2Z2} also extends to the case where the set $\mathcal{I}_{1}$ contains $\infty$, with the same $\Delta$ terms, but with $\sigma_{g-l}(\mathcal{I}_{1})$ replaced by $\sigma_{g-l-1}(\mathcal{I}_{1}\setminus\{\infty\})$ for every $1 \leq g \leq l$ (in particular, the latter expression vanishes for $l=g$). To see this, we apply Lemma \ref{thetaquotient} with $k\not\in\mathcal{I}_{1}$ such that an appropriate choice gives us essentially $\frac{\theta[e(\tilde{\mathcal{I}}_{1})](\zeta,\tau)}{\theta[e(\mathcal{I}_{1})](\zeta,\tau)}$ near $\zeta=0$ for $\tilde{\mathcal{I}}_{1}=\mathcal{I}_{1}\cup\{k\}\setminus\{\infty\}$, and substitute the solution to the Jacobi inversion problem for $\zeta$ being a small number $\delta$ times the $s$-th standard vector. It involves $Q_{h}=P_{\infty}$ for some $1 \leq h \leq g$, so that the corresponding coordinate has to be inverted, and we expand all the terms in powers of $\delta$ (note that the terms $\lambda_{k}-z(Q_{m})$ and $1/z(Q_{h})$ are quadratic in $\delta$, since the first derivatives like $(\mathcal{A}^{-1}C)_{ts}$ will vanish as well). Since the two theta constants vanish, we get a quotient of theta derivatives, and as $\frac{\partial}{\partial\zeta_{s}}\theta[e(\tilde{\mathcal{I}}_{1})](\zeta,\tau)\big|_{\zeta=0}$ is given by Theorem \ref{Thomae2Z2}, an elementary calculation proves the desired result. We also remark that when $f$ has degree $2g+2$ in Equation \eqref{curvehyperel} (rather than $2g+1$), and the base point $P_{0}$ has a finite $z$-value $\lambda_{0}$, some additional terms enter the proof of Theorem \ref{Thomae2Z2}, but its final result remains the same (regardless of whether the index 0 of the base point lies in $\mathcal{I}_{1}$ or not). We end this section by noting that these methods can be used for determining similar expressions for the derivatives of the $m$th order of the theta function associated with the partition $\mathcal{I}_{m}\cup\mathcal{I}_{m}$ from Equation \eqref{partition} for any $m$---see \cite{[Be]}.

\section{Thomae's Formula for Nonsignular $Z_{3}$ Curves \label{Deg3Basic}}

In this section we recall the usual Thomae formula for cyclic nonsingular covers of order 3, and prove some preliminaries that will be required later. Let $X$ be the complete curve given by the equation
\begin{equation}
w^{3}=\prod_{i=1}^{3q-1}(z-\lambda_{i}),\mathrm{\ with\ }\lambda_{i}\neq\lambda_{j}\mathrm{\ when\ }i \neq j. \label{Z3curve}
\end{equation}
which is a compact Riemann surface of genus $g=3q-2$. The space of holomorphic differentials on $X$ is described in \cite{[N]}, \cite{[FZ]}, \cite{[Z]}, \cite{[KZ1]} and others as follows.
\begin{prop}
A basis for the holomorphic differentials on $X$ is given by \[\tfrac{z^{l-1}\mathrm{d}z}{w^{2}}\mathrm{\ with\ }1 \leq l\leq2q-1\quad\mathrm{and}\quad\tfrac{z^{l-2q-2}\mathrm{d}z}{w}\mathrm{\ with\ }2q \leq l \leq 3q-2=g.\] \label{holdif}
\end{prop}
Proposition \ref{holdif} is easily proved by divisor computations. The normalization of $l$ in the second part will be clearer in Lemma \ref{pardersmult} below and afterwards. To explain the Thomae formula for $Z_{3}$ curves of this type one first has to consider the divisors for which this formula can be defined. Let $P_{i} \in X$ the be branch point lying over $\lambda_{i}$ for $1 \leq i\leq3q-1$, denote once again the branch point over $\infty$ by $P_{\infty}$, and consider (in correspondence with the case $m=0$ in Equation \eqref{partition}) a partition $\mathbf{\Lambda}$ of this set of indices (including $\infty$) into 3 sets $\Lambda_{0}$, $\Lambda_{1}$, and $\Lambda_{2}$, all of which having the same cardinality $q$, and with $\infty\in\Lambda_{2}$. To $\mathbf{\Lambda}$ we associate, as in Equation \eqref{charIm}, the characteristic
\begin{equation}
e_{\mathbf{\Lambda}}=\sum_{i\in\Lambda_{1}}u_{P_{\infty}}(P_{i})+2\sum_{i\in\Lambda_{2}}u_{P_{\infty}}(P_{i})+K_{P_{\infty}}, \label{charLambda}
\end{equation}
where $K_{P_{\infty}}$ is once again the vector of Riemann constants. We set $\Delta(\Lambda_{0})$, $\Delta(\Lambda_{1})$, and $\Delta(\Lambda_{2})$ to be as in Equation \eqref{deltas}, and for two different sets we define
\begin{equation}
\Delta(\Lambda_{0},\Lambda_{1})=\prod_{i\in\Lambda_{0}}\prod_{j\in\Lambda_{1}}(\lambda_{i}-\lambda_{j}),\mathrm{\ and\ similarly\ for\ }\Delta(\Lambda_{1},\Lambda_{2})\mathrm{\ and\ }\Delta(\Lambda_{2},\Lambda_{0}) \label{Delta2}
\end{equation}
(with the index $\infty$ excluded wherever $\Lambda_{2}$ is involved). Once again we define $C$ and $\tau$ as in Equation \eqref{matrices} (with the basis used for $C$ being now the one from Proposition \ref{holdif}), and we cite the following theorem from \cite{[N]}, \cite{[BR2]}, \cite{[Ko]}, and others.
\begin{thm}
The theta constant $\theta[e_{\mathbf{\Lambda}}]$ is non-vanishing, and its value is described by \[\theta[e_{\mathbf{\Lambda}}]=\alpha\epsilon_{\mathbf{\Lambda}}\sqrt{\det C}\cdot\Delta(\Lambda_{0})^{1/2}\Delta(\Lambda_{1})^{1/2}\Delta(\Lambda_{2})^{1/2}\Delta(\Lambda_{0},\Lambda_{1})^{1/6}\Delta(\Lambda_{1},\Lambda_{2})^{1/6}\Delta(\Lambda_{2},\Lambda_{0})^{1/6}.\] Here $\alpha$ is a global constant on the moduli space of such $Z_{3}$ curves, and $\epsilon_{\mathbf{\Lambda}}$ is a 12th root of unity that depends only on the partition $\mathbf{\Lambda}$. \label{Thomae1Z3}
\end{thm}


\smallskip

Our proof of Theorem \ref{Thomae2Z2} involved Lemma \ref{dQrdzetas}, namely evaluating the matrix of derivatives of the $z$-values of the solution to the Jacobi problem in terms of the local variable $\zeta=\sum_{r=1}^{g}u_{P_{\infty}}(Q_{r})+K_{P_{\infty}}$, in terms of inverting the matrix of the derivatives $\frac{\partial\zeta_{s}}{\partial z(Q_{r})}$. At that point we have ignored the fact that the natural space on which this $\zeta$ is defined is not the $g$-fold product $X^{g}$, but rather the symmetric power $\mathrm{Sym}^{g}X$. As the projection of the former onto the latter is a local diffeomorphism when the values $z(Q_{r})$, $1 \leq r \leq g$ are all distinct, this did not affect the result in the hyper-elliptic case. On the other hand, here the points $P_{i}$ with $i\in\Lambda_{2}$ appear with multiplicity 2 in the characteristic appearing in Equation \eqref{charLambda}, so we do have to consider the coordinates on $\mathrm{Sym}^{g}X$, or more precisely the derivatives with respect to them.

The general result with multiplicities is given in terms of symmetric functions, but here, with multiplicity 2, we can show in an elementary manner how these derivatives work. Let $\varphi$ be a \emph{symmetric} function of two variables, namely it satisfies $\varphi(z,w)=\varphi(w,z)$ for every $z$ and $w$. Then $\varphi$ can be expressed via the variables $u=z+w$ and $v=\frac{1}{2}(z-w)^{2}$ (typically one takes $u$ and $p=zw$, but for our purposes $v=\frac{u^{2}-p}{2}$ will be more useful), namely $\varphi(z,w)=\phi(u,v)$ with these $u$ and $v$. In this case we have the following result.
\begin{lem}
For $\varphi(t,s)=\phi(\alpha,\beta)$ with $\alpha=t+s$ and $\beta=(t-s)^{2}\big/2$ we have
\[\tfrac{\partial\phi}{\partial\alpha}=\tfrac{1}{2}\big(\tfrac{\partial\varphi}{\partial t}+\tfrac{\partial\varphi}{\partial s}\big),\quad\mathrm{while\ }\tfrac{\partial\phi}{\partial\beta}\mathrm{\ is\ }\tfrac{1}{2(t-s)}\big(\tfrac{\partial\varphi}{\partial t}-\tfrac{\partial\varphi}{\partial s}\big)\mathrm{\ if\ }t \neq s\mathrm{\ and\ }\tfrac{1}{2}\big(\tfrac{\partial^{2}\varphi}{\partial t^{2}}-\tfrac{\partial^{2}\varphi}{\partial t\partial s}\big)\mathrm{\ when\ }t=s.\] \label{symfuncder}
\end{lem}

\begin{proof}
The chain rule gives
\[\tfrac{\partial\varphi}{\partial t}=\tfrac{\partial\phi}{\partial\alpha}+(t-s)\tfrac{\partial\phi}{\partial\beta}\quad\mathrm{and}\quad\tfrac{\partial\varphi}{\partial s}=\tfrac{\partial\phi}{\partial\alpha}-(t-s)\tfrac{\partial\phi}{\partial\beta},\] from which the first two asserted formulae follow. Taking the second derivatives shows that $\frac{\partial^{2}\varphi}{\partial t^{2}}$, $\frac{\partial^{2}\varphi}{\partial s^{2}}$, and $\frac{\partial^{2}\varphi}{\partial t\partial s}$ equal, respectively,
\[\tfrac{\partial^{2}\phi}{\partial\alpha^{2}}+2(t-s)\tfrac{\partial^{2}\phi}{\partial\alpha\partial\beta}+\tfrac{\partial\phi}{\partial\beta}+(t-s)^{2}\tfrac{\partial^{2}\phi}{\partial\beta^{2}},\quad
\tfrac{\partial^{2}\phi}{\partial\alpha^{2}}-2(t-s)\tfrac{\partial^{2}\phi}{\partial\alpha\partial\beta}+\tfrac{\partial\phi}{\partial\beta}+(t-s)^{2}\tfrac{\partial^{2}\phi}{\partial\beta^{2}},\] and
\[\tfrac{\partial^{2}\phi}{\partial\alpha^{2}}-\tfrac{\partial\phi}{\partial\beta}-(t-s)^{2}\tfrac{\partial^{2}\phi}{\partial\beta^{2}},\quad\mathrm{so\ that}\quad\tfrac{\partial^{2}\varphi}{\partial s^{2}}+\tfrac{\partial^{2}\varphi}{\partial t^{2}}-2\tfrac{\partial^{2}\varphi}{\partial s\partial t}=4\tfrac{\partial\phi}{\partial\beta}+4(t-s)^{2}\tfrac{\partial^{2}\phi}{\partial\beta^{2}}.\] Since when $s=t$ the last term here vanishes, and $\frac{\partial\varphi}{\partial s}=\frac{\partial\varphi}{\partial t}$ and $\frac{\partial^{2}\varphi}{\partial s^{2}}=\frac{\partial^{2}\varphi}{\partial t^{2}}$ by the symmetry of $\varphi$, the remaining equality also follows. This proves the lemma.
\end{proof}

We can now obtain our equivalent of Lemma \ref{dQrdzetas}, which we shall prove with the coordinates $t_{r}(Q_{r})$, which now satisfy
\begin{equation}
t_{r}^{3}=z(Q_{r})-\lambda_{i}\quad\mathrm{as\ well\ as}\quad w(Q_{r})=t_{r}\big(f'(\lambda_{i})^{1/3}+O(t_{r}^{2})\big) \label{trzr3}
\end{equation}
(analogously to Equation \eqref{trzr2}) since we will consider these derivatives at the branch points. On the other hand, we shall assume that $\{Q_{r}\}_{r=1}^{2q-1}$ will be distinct branch points, but $Q_{r+2q-1}$ will be the same branch point as $Q_{r}$ for $1 \leq r \leq q-1$. Therefore the local coordinates on $\mathrm{Sym}^{g}X$ (on which our symmetric function $\zeta$ is defined) near our point will be
\begin{eqnarray}
\nonumber
\big\{\alpha_{r}=t_{r}(Q_{r})+t_{r}(Q_{r+2q-1})\big\}_{r=1}^{q-1},\qquad\mathrm{followed\ by}\qquad\{\alpha_{r}=t_{r}(Q_{r})\}_{r=q}^{2q-1}, \\
\mathrm{and\ then}\qquad\big\{\beta_{r}=\big(t_{r}(Q_{r})-t_{r}(Q_{r+2q-1})\big)^{2}\big/2\big\}_{r=1}^{q-1}. \label{coordinates}
\end{eqnarray}
We also define, in analogy with Equation \eqref{polprod}, the polynomials
\begin{equation}
F_{+}(z)=\prod_{r=1}^{2q-1}\big(z-z(Q_{r})\big)\quad\mathrm{and}\quad F_{-}(z)=\prod_{r=1}^{q-1}\big(z-z(Q_{r})\big),\quad\mathrm{with\ each\ }z(Q_{r})\mathrm{\ being\ some\ }\lambda_{i}, \label{twoprods}
\end{equation}
and with the coefficients involving $\sigma_{p}^{(r)}(Q_{1},\ldots,Q_{2q-1})$ and $\sigma_{p}^{(r)}(Q_{1},\ldots,Q_{q-1})$ as above. We can now prove the following result.
\begin{lem}
Write $\zeta=\sum_{r=1}^{g}u_{P_{\infty}}(Q_{r})$, and assume that each point $Q_{r}$ lies near a branch point $P_{i} \neq P_{\infty}$, such that the first $2q-1$ branch points are all distinct, while the remaining $q-1$ branch points form a subset of that set of branch points, ordered as above. Then the derivatives of the coordinates from Equation \eqref{coordinates} with respect to the $\zeta_{s}$s at the $\zeta$-image of our point are given by the formulae
\[\frac{\partial\alpha_{r}}{\partial\zeta_{s}}=\tfrac{f'(\lambda_{i})^{2/3}}{3F_{+}'(\lambda_{i})}\sum_{l=1}^{2q-1}(-1)^{2q-1-l}\sigma_{2q-1-l}^{(r)}(Q_{1},\ldots,Q_{2q-1})C_{ls}\quad\mathrm{for}\quad1 \leq r \leq2q-1\] and \[\frac{\partial\beta_{r}}{\partial\zeta_{s}}=\tfrac{2f'(\lambda_{i})^{1/3}}{3F_{-}'(\lambda_{i})}\sum_{l=2q}^{3q-2}(-1)^{3q-2-l}\sigma_{3q-2-l}^{(r)}(Q_{1},\ldots,Q_{q-1})C_{ls}\quad\mathrm{for}\quad2q \leq r \leq3q-2=g.\] \label{pardersmult}
\end{lem}

\begin{proof}
As in the proof of Lemma \ref{dQrdzetas}, we first evaluate the derivatives of the $\zeta_{s}$s. Sticking to points in which $Q_{r}=Q_{r+2q-1}$ for every $1 \leq r \leq q-1$, Lemma \ref{symfuncder} shows that the first $2q-1$ columns of the derivatives of $\zeta$ are the usual derivatives with respect to the first $2q-1$ points, and in the remaining $q-1$ columns we must put derivatives of second order of $\zeta$. Once again we shall begin by evaluating the derivatives of $\sum_{s=1}^{g}C_{ls}\zeta_{s}$, where the derivatives of the first order will give just
\begin{equation}
\tfrac{z(Q_{r})^{l-1}}{w(Q_{r})^{2}}\cdot\tfrac{\mathrm{d}z(Q_{r})}{\mathrm{d}t_{r}(Q_{r})}=\tfrac{3\lambda_{i}^{l-1}}{f'(\lambda_{i})^{2/3}}+O(t_{r}^{3})\quad\mathrm{or}\quad
\tfrac{z(Q_{r})^{l-2q-2}}{w(Q_{r})}\cdot\tfrac{\mathrm{d}z(Q_{r})}{\mathrm{d}t_{r}(Q_{r})}=\tfrac{3t_{r}\lambda_{i}^{l-2k-2}}{f'(\lambda_{i})^{1/3}}+O(t_{r}^{4}), \label{firstders}
\end{equation}
according to whether $1 \leq l\leq2q-1$ or $2q \leq l \leq 3q-2=g$. As the derivative with respect to $t_{r}(Q_{r})$ depends only on $Q_{r}$ and not on $Q_{r+2q-1}$, the mixed derivatives from Lemma \ref{pardersmult} vanish (this is also visible in the fact that $\zeta$ is the sum of functions, each one depending only on a single point $Q_{r}$), so that only the pure second derivative has to be taken into account.

Now, differentiating the terms with small $l$ in Equation \eqref{firstders} with respect to $t_{r}$ annihilates the constant and gives just $O(t_{r}^{2})$ from the derivative of the error term, while for large $l$ the derivative will leave the constant without the multiplier $t_{r}$ (up to $O(t_{r}^{3})$ from the derivative of the error term). Moreover, we are interested in the point where each $Q_{r}$ equals the corresponding $P_{i}$, i.e., where each $t_{r}$ vanishes. Therefore in Equation \eqref{firstders} the error terms and the full expression with large $l$ vanish, while the second derivatives for small $l$ vanish as well and the constant for large $l$ remains. The resulting matrix $\mathcal{A}$ is a block matrix, with a block of size $(2q-1)\times(2q-1)$ followed by a block of size $(q-1)\times(q-1)$, where the two blocks are Vandermonde matrices modified as in Lemma \ref{dQrdzetas}. In addition, the derivatives of $\zeta$ with respect to the coordinates from Equation \eqref{coordinates} is $C^{-1}\mathcal{A}$ as in Lemma \ref{dQrdzetas}, so that the required derivatives are the entries of $\mathcal{A}^{-1}C$. The proof of Lemma \ref{dQrdzetas} now shows that $\mathcal{A}^{-1}$ is a (block) matrix whose $rl$-entry is
\[\tfrac{f'(\lambda_{i})^{2/3}}{3F_{+}'(\lambda_{i})}(-1)^{2k-1-l}\sigma_{2q-1-l}^{(r)}(Q_{1},\ldots,Q_{2q-1}),\ \tfrac{2f'(\lambda_{i})^{1/3}}{3F_{-}'(\lambda_{i})}(-1)^{3k-2-l}\sigma_{3q-2-l}^{(r-2k-1)}(Q_{1},\ldots,Q_{q-1}),\] or 0, according to whether \[1 \leq r,l \leq 2q-1,\quad2q \leq r,l \leq 3q-2,\quad\mathrm{or\ one\ index\ is\ small\ and\ one\ is\ large,\ respectively}.\] Substituting these values into the expansion of $(\mathcal{A}^{-1}C)_{rs}$ in the two possible ranges for $r$, and recalling that $\beta_{r}$ corresponds to the $(r+2k-1)$st row, yields the desired expressions. This completes the proof of the lemma.
\end{proof}

\section{The Thomae Derivative Formula for Trigonal Curves \label{ThomaeDeg3}}

In this section we prove a Thomae derivative formula for trigonal curves with smooth affine models (these curves are called \emph{non-singular $Z_{3}$ curves} in \cite{[FZ]}). Fix representing paths for a canonical homology basis $\{a_{j}\}_{j=1}^{g}$ and $\{b_{j}\}_{j=1}^{g}$ on a general compact Riemann surface $X$ of genus $g$, so that all the integrals that follow will be understood to go over these representing paths. Recall that given two distinct points $R$ and $Q$ on $X$, neither of which lies on any representing path, there exists a differential that is holomorphic on $X\setminus\{R,Q\}$ and such that it has simple poles at $R$ and $Q$, with respective residues 1 and $-1$. This differential, which we denote by $\Omega_{R,Q}$, is uniquely determined by the normalization $\int_{a_{j}}\Omega_{R,Q}=0$ for every $1 \leq j \leq g$. Moreover, if $S$ and $T$ is another pair of such points, then the Riemann bilinear relations state that \[\int_{Q}^{R}\Omega_{S,T}=\int_{T}^{S}\Omega_{R,Q},\mathrm{\ provided\ that\ the\ integration\ paths\ do\ not\ intersect\ }a_{j}\mathrm{\ and\ }b_{j}.\] Uniqueness implies that given 3 such points $R$, $Q$, and $S$ we have
\begin{equation}
\Omega_{R,Q}+\Omega_{Q,S}=\Omega_{R,S},\mathrm{\ so\ that\ for\ }D=\sum_{r=1}^{l}R_{r}\mathrm{\ and\ }\Delta=\sum_{r=1}^{l}Q_{r}\mathrm{\ we\ set\ }\Omega_{D,\Delta}=\sum_{r=1}^{l}\Omega_{R_{r},Q_{r}}, \label{Omegadivs}
\end{equation}
and the result is independent of the orderings. Moreover, let a differential $\omega$ be given, and take two divisors $\Xi=\sum_{k=1}^{p}S_{k}$ and $\Gamma=\sum_{k=1}^{p}T_{k}$, where no point in neither support is on some $a_{j}$ or some $b_{j}$. In this situation we write \[\int_{\Gamma}^{\Xi}\omega=\sum_{k=1}^{p}\int_{T_{k}}^{S_{k}}\omega,\mathrm{\ and\ deduce\ that\ }\int_{\Delta}^{D}\Omega_{\Xi,\Gamma}=\int_{\Gamma}^{\Xi}\Omega_{D,\Delta},\mathrm{\ under\ the\ same\ condition}.\]

For establishing our equivalent of Lemma \ref{thetaquotient} we shall be using the following result, which is essentially already proved in Sections 154 and 171 of \cite{[Ba]}. Let $z$ be a meromorphic function of degree $n$ on $X$, and assume that none of its poles lie on the representing paths. For $\lambda\in\mathbb{CP}^{1}$ write $z^{*}\lambda$ for the divisor on $X$ consisting of the pre-images of $\lambda$ in $X$, counted with multiplicities. We now have the following relation.
\begin{lem}
Given $\lambda\in\mathbb{C}$ such that none of its pre-images under $z$ lie on the paths chosen to represent the homology basis, there exists a vector $\nu_{\lambda}\in\mathbb{Z}^{g}$ such that for two divisors $D$ and $\Delta$ as in Equation \eqref{Omegadivs} we have the equality \[\exp\bigg(\int_{z^{*}\infty}^{z^{*}\lambda}\Omega_{\Delta,D}\bigg)\cdot\mathbf{e}[\nu_{\lambda}^{t}u(D-\Delta)]=\prod_{r=1}^{l}\frac{\lambda-z(Q_{r})}{\lambda-z(R_{r})}.\] \label{dif3kind}
\end{lem}
We remark that Lemma \ref{dif3kind} with $l=g$ can be used for solving the Jacobi inversion problem.

\begin{proof}
The proof of Abel's Theorem in \cite{[FK]}, in the explicit case relating the function $z-\lambda$ and its divisor $z^{*}\lambda-z^{*}\infty$, implies that there are integers $\{\nu_{\lambda,s}\}_{s=1}^{g}$ such that
\[\textstyle{\tfrac{dz}{z-\lambda}=\Omega_{z^{*}\lambda,z^{*}\infty}-2\pi\imath\sum_{s=1}^{g}\nu_{\lambda,s}v_{s},\mathrm{\ and\ set\ }\nu_{\lambda}\in\mathbb{Z}^{g}\mathrm{\ to\ be\ with\ coordinates\ }\nu_{\lambda,s},1 \leq s \leq g.}\] Integrating from $R_{r}$ to $Q_{r}$, summing over $r$, applying the Riemann bilinear relations, and exponentiating yields the desired equality. This proves the lemma.
\end{proof}
We remark that the vector $\nu_{\lambda}$ from Lemma \ref{dif3kind} depends continuously on $\lambda$, so that it is constant on domains in $\mathbb{CP}^{1}$ (these domains are separated by the images of the representing paths in $X$ under $z$). In particular, it is clear that $\nu_{\infty}=0$ (just take the limit $\lambda\to\infty$ in Lemma \ref{dif3kind}), so that $\nu_{\lambda}=0$ for every large enough $\lambda$. However, in our application below it may not be possible to organize $\lambda$ to be large enough, so that $\nu_{\lambda}$ may be non-zero.

\smallskip

Following Section 187 of \cite{[Ba]} we present a formula, using theta functions, for expressions like the one on the left hand side of Lemma \ref{dif3kind}, for divisors of degree $g$. The result, which is interesting in its own right, still holds for a general Riemann surface.
\begin{thm}
Let $D$ and $\Delta$ be divisors of degree $g$ on $X$, take a base point $P_{0}$ on $X$, set \[e=u_{P_{0}}(\Delta)+K_{P_{0}}\mathrm{\ and\  }\mathfrak{e}=u_{P_{0}}(D)+K_{P_{0}},\quad\mathrm{and\ assume\ that}\quad\theta(e)\neq0\mathrm{\ and\ }\theta(\mathfrak{e})\neq0.\] Then we have, for every $R$ and $Q$ in $X$, the equality \[\exp\bigg(\int_{Q}^{R}\Omega_{D,\Delta}\bigg)=\frac{\theta\big(u_{P_{0}}(R)-\mathfrak{e},\tau\big)\theta\big(u_{P_{0}}(Q)-e,\tau\big)}{\theta\big(u_{P_{0}}(R)-e,\tau\big)\theta\big(u_{P_{0}}(Q)-\mathfrak{e},\tau\big)}.\] \label{thetaOmega}
\end{thm}

\begin{proof}
Denote by $\xi(R,Q)$ the quotient arising from dividing the right hand side by the left hand side. We claim that for fixed $Q$, this ratio is a well-defined constant function of $R \in X$ (as long as neither of the theta terms involving $Q$ vanish), and vice versa. This amounts to showing that its divisor is trivial, and that the expression is invariant under the monodromy action.

We begin with the monodromy. We recall from \cite{[FK]} that $\Omega_{D,\Delta}$ is normalized to have vanishing $a_{j}$-integrals, and that Equation \eqref{periodicity} implies that the theta function is invariant under changing the argument by a cycle $a_{j}$. Hence both parts of the quotient $\xi(P,Q)$ are invariant under this part of the monodromy. On the other hand, we know that \[\int_{b_{j}}\Omega_{D,\Delta}=2\pi\imath\int_{\Delta}^{D}v_{j},\mathrm{\ and\ }b_{j}\mathrm{\ multiplies\ }\theta\big(u_{P_{0}}(R)-e,\tau\big)\mathrm{\ by\ }\mathbf{e}\big[-\tfrac{\varepsilon_{j}^{t}\tau\varepsilon_{j}^{t}}{2}-\varepsilon_{j}^{t}\big(u_{P_{0}}(R)-e\big)\big],\] where here $\varepsilon_{j}$ stands for the $j$th standard basis vector of length $g$. The quotient on the right hand side is therefore multiplied by $\mathbf{e}[\varepsilon_{j}^{t}(\mathfrak{e}-e)]$, which is the exponent of $2\pi\imath\int_{\Delta}^{D}v_{j}$ by definition, and $\xi(R,Q)$ is invariant under this operation as well. As for the divisor in the variable $R$, the right hand side has divisor $D-\Delta$ by the Riemann Vanishing Theorem, and the left hand side has the same divisor because of the singularities of $\Omega_{D,\Delta}$. Hence $R\mapsto\xi(R,Q)$ is a well-defined function on $X$ with a trivial divisor, hence a constant.

Now, the same argument (with some signs inverted) show that $Q\mapsto\xi(R,Q)$ is also constant for every $R$ (up to a few singular points perhaps). For evaluating the constant we substitute $R=Q$, where the two asserted expressions reduce to 1. Hence $\xi(R,Q)=1$ for every $R$ and $Q$ in $X$, which proves the theorem.
\end{proof}

Lemma \ref{thetaquotient} can be proved by combining Lemma \ref{dif3kind} and Theorem \ref{thetaOmega} for a hyper-elliptic curve, in a similar manner to the proof of the next result. In our case, where $z:X\to\mathbb{CP}^{1}$ is a cyclic cover of degree 3 as in Equation \eqref{Z3curve}, we do it and arrive at the following conclusion.
\begin{prop}
Let $D=\sum_{r=1}^{g}R_{r}$ and $\Delta=\sum_{r=1}^{g}Q_{r}$ be divisors on $X$ (as in Equation \eqref{Omegadivs} with $l=g$), set $e$ and $\mathfrak{e}$ to be as in Theorem \ref{thetaOmega} with the base point $P_{0}$ being the pole $P_{\infty}$ of $z$, and denote by $\nu_{k}$ the vector $\nu_{\lambda_{k}}\in\mathbb{Z}^{g}$ that is associated with $\lambda_{k}$ in Lemma \ref{dif3kind}. We then have the equality \[\frac{\theta^{3}\big(u_{P_{\infty}}(P_{k})-e,\tau\big)\theta^{3}(-\mathfrak{e},\tau)}{\theta^{3}\big(u_{P_{\infty}}(P_{k})-\mathfrak{e},\tau\big)\theta^{3}(-e,\tau)}\cdot\mathbf{e}[\nu_{k}^{t}(\mathfrak{e}-e)]=
\prod_{r=1}^{g}\frac{\lambda_{k}-z(Q_{r})}{\lambda_{k}-z(R_{r})}.\] \label{quottheta}
\end{prop}

\begin{proof}
Lemma \ref{dif3kind} compares the right hand side with $\exp\big(\int_{z^{*}\infty}^{z^{*}\lambda_{k}}\Omega_{\Delta,D}\big)\cdot\mathbf{e}[\nu_{k}^{t}(\mathfrak{e}-e)]$, by the definition of $e$, $\mathfrak{e}$, and $\nu_{k}$ (indeed, we can cancel $K_{P_{0}}$ and then $u_{P_{0}}(D-\Delta)$ is independent of $P_{0}$). On the other hand, the divisors $z^{*}\lambda_{k}$ and $z^{*}\infty$ are just $3P_{k}$ and $3P_{\infty}$ respectively, so that the exponent is $\exp\big(\int_{P_{\infty}}^{P_{k}}\Omega_{\Delta,D}\big)^{3}$. But this exponent is the left hand side in Theorem \ref{thetaOmega}, with $Q=P_{0}=P_{\infty}$ (so that $u_{P_{0}}(Q)$ vanishes) and $R=P_{k}$. We therefore substitute the right hand side from that theorem instead, which yields the asserted quotient of theta functions because of the third power involved. This proves the proposition.
\end{proof}
In this context we remark that $u_{P_{\infty}}(P_{k})$ is a point of order 3 in the Jacobian, since $3P_{k}-3P_{\infty}$ is the (principal) divisor of $z-\lambda$. Moreover, it is never a half-characteristic (when $k\neq\infty$), since otherwise $2P_{k}-2P_{\infty}$ would have also been principal, and therefore so will $P_{k}-P_{\infty}$ be, and no divisor of the form $P-Q$ can be principal on any compact Riemann surface with positive genus.

\begin{rmk}
In the case where $q=1$ (hence $g=1$ as well), and $z$ is normalized such that $\lambda_{2}=0$, Equation \eqref{Z3curve} for the trigonal curve becomes $w^{3}=z(z-\lambda_{1})$. In this case Proposition \ref{quottheta} can also be established as follows. Denote the standard 3rd root of unity $\mathbf{e}\big(\frac{1}{3}\big)$ by $\rho$, and let $L$ be the scalar multiple of the lattice $\mathbb{Z}\oplus\mathbb{Z}\rho$ whose $g_{3}$-parameter is $-\lambda_{1}^{2}$ (and $g_{2}=0$, of course). Then if $E$ is the elliptic curve $\mathbb{C}/L$, of $j$-invariant 0 and an automorphism group of order 6, and $\wp_{L}$ is the corresponding Weierstrass $\wp$-function associated with the lattice $L$, then the map sending $\xi \in E$ to $(z,w)$ with $w=\wp_{L}(\xi)$ and $z=\frac{\wp_{L}'(\xi)+\lambda_{1}}{2}$ defines, by the classical theory of complex elliptic curves, an isomorphism from $E$ onto $X$ (the trigonal automorphism group, acting on $w$ alone, becomes, via Equation (5.3) of \cite{[EMO]}, complex multiplication by the group generated by $\rho$, and adding the automorphism $z\mapsto\lambda_{1}-z$ completes it to the group of order 6). If $c$ is the scalar differentiating $L$ from $\mathbb{Z}\oplus\mathbb{Z}\rho$ (whose 6th power is determines by the value of $g_{3}$), and the homology basis for $X$ is chosen such that the $1\times1$ matrix $\tau$ from Equation \eqref{matrices} is $\rho$, then $u_{P_{\infty}}$ identifies $X$ with $J(X)=\mathbb{C}/(\mathbb{Z}\oplus\mathbb{Z}\rho)$ such that for $\xi \in E$ we have the explicit equality $u_{P_{\infty}}\big(z(\xi),w(\xi)\big)=\frac{\xi}{c}$. For such curves Proposition 5.1 of \cite{[EMO]} established the generalized addition formula \[\frac{\sigma_{L}(\xi-\eta)\sigma_{L}(\xi-\rho\eta)\sigma_{L}(\xi-\overline{\rho}\eta)}{\sigma_{L}(\xi)^{3}\sigma_{L}(\eta)^{3}}=z(\xi)-z(\eta)\] for $\xi$ and $\eta$ in $\mathbb{C}$, where $\sigma_{L}(\xi)=\xi\prod_{0 \neq l \in L}\big(1-\frac{\xi}{l}\big)e^{\frac{\xi}{l}+\frac{\xi^{2}}{2l^{2}}}$ is the Weierstrass $\sigma$-function associated with the lattice $L$. Note that both sides are well-defined for $\xi$ and $\eta$ in the quotient $E$, and by the homogeneity property of $\sigma$ we can replace each term like $\sigma_{L}(\xi)$ by $\sigma_{\rho}(\frac{\xi}{c})$ (with an extra denominator of $c^{3}$), where for $\tau\in\mathcal{H}_{1}$ the function $\sigma_{\tau}$ is the $\sigma$-function associated with the lattice $\mathbb{Z}\oplus\mathbb{Z}\tau$. The latter function behaves under adding a lattice element via \[\sigma_{\tau}(\zeta+n\tau+l)=e^{G_{2}(\tau)\frac{(n\tau+l)^{2}}{2}-\pi\imath n^{2}\tau+[G_{2}(\tau)(n\tau+l)-2\pi\imath n]\zeta}(-1)^{n+l}\sigma_{\tau}(\zeta)\qquad\mathrm{for}\qquad\zeta\in\mathbb{C}\] when $n$ and $l$ are integers, and $G_{2}$ is the quasi-modular Eisenstein series of weight 2. Taking $\eta \in E$ to be the point with $z(\eta)=\lambda_{1}$ and $w(\eta)=0$, which is thus invariant under the trigonal automorphism group, the $\sigma$-product in the numerator becomes $\sigma(\xi-\eta)^{3}$ times a non-trivial exponent, arising from the fact that the image of $\eta$ in $E$ is $\rho$-invariant, while $\eta\in\mathbb{C}$ itself is not. If we normalize the scalar $c$ such that $\frac{\eta}{c}=\frac{\imath}{\sqrt{3}}$ (this element and its additive inverse are the only elements whose images modulo $\mathbb{Z}\oplus\mathbb{Z}\rho$ are non-trivial and $\rho$-invariant, up to adding lattice elements), then since $\rho\frac{\imath}{\sqrt{3}}=\frac{\imath }{\sqrt{3}}-\rho-1$ and $\overline{\rho}\frac{\imath }{\sqrt{3}}=\frac{\imath }{\sqrt{3}}-\rho$ we obtain, via the homogeneity property of $\sigma$ and its behavior under adding lattice elements to the argument, that the exponential differentiating the $\sigma$-product in the numerator from $\sigma(\xi-\eta)^{3}$ is \[e^{G_{2}(\rho)\frac{(\rho+1)^{2}}{2}-\pi\imath\rho+[G_{2}(\rho)(\rho+1)-2\pi\imath]\frac{\xi-\eta}{c}}(+1) \cdot e^{G_{2}(\rho)\frac{\rho^{2}}{2}-\pi\imath\rho+[G_{2}(\rho)\rho-2\pi\imath]\frac{\xi-\eta}{c}}(-1).\] Now, since $(\rho+1)^{2}=\rho$, $\rho^{2}=\overline{\rho}$, $\rho+\overline{\rho}=-1$, $2\rho+1=\imath\sqrt{3}$, and the quasi-modularity of $G_{2}$ under the stabilizer of $\rho$ determines the value of $G_{2}(\rho)$ to be $-\frac{2\pi}{\sqrt{3}}$, we can substitute the value $\frac{\imath}{\sqrt{3}}$ of $\frac{\eta}{c}$ and the latter factor reduces to \[-e^{\frac{\pi}{\sqrt{3}}+\pi\imath+\pi\sqrt{3}-\big[4\pi\imath+\frac{2\pi}{\sqrt{3}}\cdot\imath\sqrt{3}\big]\big(\frac{\xi}{c}-\frac{\imath}{\sqrt{3}}\big)}=+e^{-\frac{2\pi}{\sqrt{3}}-6\pi\imath\frac{\xi}{c}}.\] Since $z(\eta)=\lambda_{1}$ by assumption, and $\frac{\xi}{c}=u_{P_{\infty}}\big(z(\xi),w(\xi)\big)$ by our normalization, it follows that if we denote the point $\big(z(\xi),w(\xi)\big)$ by $Q_{1}$ (with $z(Q_{1})=z(\xi)$, of course), then we get \[z(Q_{1})-\lambda_{1}=\frac{\sigma_{L}(\xi-\eta)^{3}e^{-\frac{2\pi}{\sqrt{3}}-6\pi\imath u_{P_{\infty}}(Q_{1})}}{\sigma_{L}(\xi)^{3}\sigma_{L}(\eta)^{3}}.\] Taking another element, say $\omega\in\mathbb{C}$, and setting $R_{1}=\big(z(\omega),w(\omega)\big)$, we deduce that the quotient $\frac{\lambda_{1}-z(Q_{1})}{\lambda_{1}-z(R_{1})}$ (which is the right hand side of Proposition \ref{quottheta}, with $k=1$, since $g=1$ in this case) can be written as $\frac{\sigma_{L}(\xi-\eta)^{3}\sigma_{L}(\omega)^{3}}{\sigma_{L}(\omega-\eta)^{3}\sigma_{L}(\xi)^{3}}\mathbf{e}\big(3u(R_{1}-Q_{1})\big)$. By the homogeneity property of the $\sigma$-function we can replace each $\sigma_{L}$ by $\sigma_{\rho}$ and divide each argument by $c$, and as the function $\sigma_{\tau}$ is related to the genus 1 theta function $\theta\big[{1 \atop 1}\big]$ by the equality \[\textstyle{\theta\big[{1 \atop 1}\big](\zeta,\tau)=\frac{d}{d\zeta}\theta\big[{1 \atop 1}\big](\zeta,\tau)\big|_{\zeta=0}e^{-G_{2}(\tau)\zeta^{2}/2}\sigma_{\tau}(\zeta)},\] we can substitute this theta function (with $\tau=\rho$) instead of each occurrence of $\sigma_{\rho}$, where the multiplying coefficients cancel and we get an additional exponent of \[e^{3G_{2}(\rho)\big(\frac{(\xi-\eta)^{2}}{2c^{2}}+\frac{\omega^{2}}{2c^{2}}-\frac{(\omega-\eta)^{2}}{2c^{2}}-\frac{\xi^{2}}{2c^{2}}\big)}=e^{3G_{2}(\rho)\cdot\frac{\eta}{c}\cdot\frac{\omega-\xi}{c}}=
e^{-3\frac{2\pi}{\sqrt{3}}\cdot\frac{\imath}{\sqrt{3}} \cdot u(R_{1}-Q_{1})}=\mathbf{e}\big(-u(R_{1}-Q_{1})\big).\] Applying Equation \eqref{transchar} to each theta function (with $\varepsilon=\delta=1$), noting that the additional exponents cancel, observing that $K_{P_{0}}$ is the image of $\frac{\tau+1}{2}$ in $\mathbb{C}/\mathbb{Z}\oplus\mathbb{Z}\tau$ regardless of $P_{0}$ in genus 1 and that $\frac{\eta}{c}=u_{P_{\infty}}(P_{1})$ by definition, and substituting the values of $e$ and $\mathfrak{e}$ arising from $\Delta=Q_{1}$ and $D=R_{1}$ as in Theorem \ref{thetaOmega}, this quotient indeed reduces to the expression from the left hand side of Proposition \ref{quottheta} (with $\nu_{1}=2$ and $\tau=\rho$), as desired. \label{q1ell}
\end{rmk}

Choosing an appropriate divisor $D$ in Proposition \ref{quottheta} yields our analogue for Lemma \ref{thetaquotient}.
\begin{prop}
Given $1 \leq k\leq3q-1$, there is a 12th root of unity $\epsilon_{k}$ such that considered as functions of $\{Q_{r}\}_{r=1}^{g}$, we get the equality \[\frac{\theta[u_{P_{\infty}}(P_{k})]^{3}\big(-\sum_{r=1}^{g}u_{P_{\infty}}(Q_{r})-K_{P_{\infty}},\tau\big)}{\theta^{3}\big(-\sum_{r=1}^{g}u_{P_{\infty}}(Q_{r})-K_{P_{\infty}},\tau\big)}=
\frac{\epsilon_{k}}{f'(\lambda_{k})}\prod_{r=1}^{g}\big(\lambda_{k}-z(Q_{r})\big).\] \label{quotforZ3}
\end{prop}

\begin{proof}
Take the divisor $D$ from Proposition \ref{quottheta} to be the one arising from omitting $P_{\infty}$ from the divisor associated with a partition $\mathbf{\Lambda}$, in which \[|\Lambda_{0}|=|\Lambda_{1}|=|\Lambda_{2}|=q,\quad\infty\in\Lambda_{2},\quad\mathrm{and}\quad k\in\Lambda_{0},\quad\mathrm{so\ that}\quad\mathfrak{e}=e_{\mathbf{\Lambda}}.\] We recall from the proof of Abel's Theorem in \cite{[FK]} that the vector $\nu_{\lambda}$ from Lemma \ref{dif3kind} appears in the equation
\[u(z^{*}\lambda-z^{*}\infty)=\tau\nu_{\lambda}+I\mu_{\lambda}\mathrm{\ for\ }\mu_{\lambda}\in\mathbb{Z}^{g},\mathrm{\ and\ with\ }\lambda=\lambda_{k}\mathrm{\ the\ left\ hand\ side\ is\ }3u_{P_{\infty}}(P_{k}).\] Then a triple application of Equation \eqref{transchar} combines with the evenness of $\theta$ to show that
\[\frac{\theta^{3}\big(u_{P_{\infty}}(P_{k})-e,\tau\big)\theta^{3}(-\mathfrak{e},\tau)}{\theta^{3}\big(u_{P_{\infty}}(P_{k})-\mathfrak{e},\tau\big)\theta^{3}(-e,\tau)}\cdot\mathbf{e}[\nu_{k}^{t}(\mathfrak{e}-e)]=
\hat{\varepsilon}\frac{\theta\big[u_{P_{\infty}}(P_{k})\big]^{3}(-e,\tau)\theta^{3}[\mathfrak{e}](0,\tau)}{\theta\big[u_{P_{\infty}}(P_{k})-\mathfrak{e}\big]^{3}(0,\tau)\theta^{3}(-e,\tau)}\] for some 3rd root of unity $\hat{\varepsilon}$. With $e=\sum_{r=1}^{g}u_{P_{\infty}}(Q_{r})+K_{P_{\infty}}$ this is the desired quotient from the left hand side, multiplied by some constants, and on the right hand side of Proposition \ref{quottheta} we get the desired product, divided by $F_{+}(\lambda_{k})F_{-}(\lambda_{k})$ in the notation of Equation \eqref{twoprods}.

Now, the calculations from \cite{[FZ]} (involving the operator denoted there by $T_{P_{k}}$) imply that
\begin{equation}
u_{P_{\infty}}(P_{k})-\mathfrak{e}=e_{\widetilde{\mathbf{\Lambda}}}\mathrm{\ for\ }\widetilde{\mathbf{\Lambda}}\mathrm{\ with\ }\widetilde{\Lambda}_{2}=\Lambda_{0}\cup\{\infty\}\setminus\{k\},\widetilde{\Lambda}_{1}=\Lambda_{1},\mathrm{\ and\ }\widetilde{\Lambda}_{0}=\Lambda_{2}\cup\{k\}\setminus\{\infty\}. \label{TPkDelta}
\end{equation}
Therefore the desired left hand side equals the product on the right hand side times the constant $\frac{\theta[e_{\widetilde{\mathbf{\Lambda}}}]^{3}(0,\tau)/\hat{\varepsilon}}{\theta[e_{\mathbf{\Lambda}}]^{3}(0,\tau)F_{+}(\lambda_{k})F_{-}(\lambda_{k})}$. We substitute the values of the two theta constants that are given in Theorem \ref{Thomae1Z3} (raised to the 3rd power), obtain another 4th root of unity, and verify that after the cancelations every term of the form $\lambda_{k}-\lambda_{i}$ for $i \neq k$ (and $i\neq\infty$) ends up appearing the denominator with the power 1. Since the resulting product is precisely $f'(\lambda_{k})$, this completes the proof of the proposition.
\end{proof}

\smallskip

The Thomae derivative formulae are associated, as in the hyper-elliptic case, with positive divisors $\Delta$ of degree $g-1$ that are supported on the branch points and whose images in $J(X)$ represent simple zeros of $\theta$. We characterize these divisors on our $Z_{3}$ curve as follows.
\begin{thm}
Let $\Xi$ be positive divisor degree $g-1$ on $X$ whose support consists only of branch points. Then $u_{P_{\infty}}(\Xi)+K_{P_{\infty}}$ lies on the non-singular locus of the theta divisor if and only if $\Xi$ is of the form $2\sum_{i\in\Lambda_{2}}P_{i}+\sum_{i\in\Lambda_{1}}P_{i}$ for a partition $\mathbf{\Lambda}$ of the set of branch point into the sets $\Lambda_{0}$, $\Lambda_{1}$, and $\Lambda_{2}$ that satisfies one of the two cardinality conditions \[\mathrm{either\ }|\Lambda_{2}|=|\Lambda_{1}|=q-1\mathrm{\ and\ }|\Lambda_{0}|=q+2,\quad\mathrm{or}\quad|\Lambda_{2}|=q-2\mathrm{\ and\ }|\Lambda_{2}|=|\Lambda_{2}|=q+1.\] \label{divsforder}
\end{thm}

\begin{proof}
The extended version of the Riemann Vanishing Theorem and the Riemann--Roch Theorem imply that for a positive divisor $\Xi$ on $X$, the zero $u_{P_{\infty}}(\Xi)+K_{P_{\infty}}$ of $\theta$ is simple if and only if there is no non-constant function whose poles are bounded by $\Xi$, or equivalently there is only one canonical divisor from which subtracting $\Xi$ yields a positive divisor. For divisors supported on the branch points (i.e., divisors that are invariant under the cyclic Galois group of $z:X\to\mathbb{CP}^{1}$), the first condition implies that no branch point may appear in $\Xi$ to order 3 or higher. Hence $\Xi$ can be represented by such a partition $\mathbf{\Lambda}$. Moreover, Lemma 5.2 of \cite{[KZ1]} implies that the differential whose divisor is larger than $\Xi$ must be either of the form $p(z)\frac{dz}{w^{2}}$ for a polynomial of degree at most $2q-1$, or of the form $p(z)\frac{dz}{w}$ where $p$ has degree at most $q-1$. Therefore in the first case we need that no divisor of the form $\mathrm{div}\big(p(z)\frac{dz}{w}\big)$ will be larger than $\Xi$, and only one divisor $\mathrm{div}\big(p(z)\frac{dz}{w^{2}}\big)$ will have that property. One easily verifies that this implies \[|\Lambda_{2}| \geq q-1\mathrm{\ and\ }|\Lambda_{1}|+|\Lambda_{2}|=2q-2,\mathrm{\ whence\ }|\Lambda_{2}|=|\Lambda_{1}|=q-1\mathrm{\ since\ }\deg\Xi=g-1=3q-3\] and we are in the first situation (because $\Lambda_{0}$ is the complement in a set of cardinality $3q$). In the second case no divisor $\mathrm{div}\big(p(z)\frac{dz}{w^{2}}\big)$ can exceed $\Xi$ and only one of the form $\mathrm{div}\big(p(z)\frac{dz}{w}\big)$ does so. Here we get \[|\Lambda_{1}|+|\Lambda_{2}|\geq2q-1\mathrm{\ and\ }|\Lambda_{2}|=q-2,\mathrm{\ so\ that\ }|\Lambda_{1}|=q+1\mathrm{\ because\ }\deg\Xi=g-1=3q-3\] again, and we are in the second situation. This proves the theorem.
\end{proof}
Note that Theorem \ref{divsforder} assumes implicitly that $q\geq2$, hence $g=3q-2\geq4$, as we shall henceforth do. The case with $q=1$ and $g=1$ is degenerate, and the only divisor $\Xi$ in that theorem is the trivial one, which is of the first type there. In that case Theorem \ref{Thomae2Z3t1} below will produce the well-known relation between the Jacobi theta derivative formula and the discriminant of the associated elliptic curve.

\smallskip

Theorem \ref{divsforder} implies that there are two types of divisors $\Xi$ for which we would like to prove a Thomae derivative formula. As in Theorem \ref{Thomae2Z2}, for each type we shall need to assume that $\infty$ lies in a particular set for our method to work. We shall later indicate, as in the end of Section \ref{HypEll}, what happens when $\infty$ lies in a different set.

We begin with the divisors of the first type, where for a partition $\mathbf{\Lambda}$ we set $\sigma_{p}(\Lambda_{1}\cup\Lambda_{2})$ to be the $p$th symmetric function based on the $z$-values of the branch points with indices in $\Lambda_{1}\cup\Lambda_{2}$ (as usual).
\begin{thm}
Let $\Xi$ be the positive divisor that is associated with the partition $\mathbf{\Lambda}$, in which
\[|\Lambda_{2}|=|\Lambda_{1}|=q-1,\quad|\Lambda_{0}|=q+2,\quad\mathrm{and}\quad\infty\in\Lambda_{0},\] and set $e_{\mathbf{\Lambda}}$ to be as in Equation \eqref{charLambda}. Then for $1 \leq s \leq g$ we get the equality
\[\frac{\partial}{\partial\zeta_{s}}\theta[e_{\mathbf{\Lambda}}](\zeta,\tau)\bigg|_{\zeta=0}=
\Delta(\Lambda_{0})^{1/2}\Delta(\Lambda_{1})^{1/2}\Delta(\Lambda_{2})^{1/2}\Delta(\Lambda_{0},\Lambda_{1})^{1/6}\Delta(\Lambda_{1},\Lambda_{2})^{1/6}\Delta(\Lambda_{2},\Lambda_{0})^{1/6}\times\] \[\times\frac{\epsilon_{\mathbf{\Lambda}}\alpha}{3}\sqrt{\det C}\sum_{l=1}^{2q-1}(-1)^{2q-1-l}\sigma_{2q-1-l}(\Lambda_{1}\cup\Lambda_{2})C_{ls},\] where $\alpha$ is the constant from Theorem \ref{Thomae1Z3}, the matrix $C$ is defined in Equation \eqref{matrices}, and $\epsilon_{\mathbf{\Lambda}}$ is a 36th root of unity. \label{Thomae2Z3t1}
\end{thm}

\begin{proof}
Choose an index $\infty \neq k\in\Lambda_{0}$, and set $\Delta$ to be the divisor corresponding to the partition $\widetilde{\mathbf{\Lambda}}$ for which
\[\widetilde{\Lambda}_{2}=\Lambda_{1}\cup\{\infty\},\ \widetilde{\Lambda}_{1}=\Lambda_{2}\cup\{k\},\mathrm{\ and\ }\widetilde{\Lambda}_{0}=\Lambda_{0}\setminus\{k,\infty\},\mathrm{\ so\ that\ }|\widetilde{\Lambda}_{2}|=|\widetilde{\Lambda}_{1}|=|\widetilde{\Lambda}_{0}|=q\] (and $\infty\in\widetilde{\Lambda}_{2}$). We take the points $Q_{r}$, $1 \leq r \leq g$ to be near the branch points constructing $\Delta$ (without $P_{\infty}$), ordered as in Lemma \ref{pardersmult}, and observe that a calculation similar to Equation \eqref{TPkDelta} compares $u_{P_{\infty}}(P_{k})-e_{\widetilde{\mathbf{\Lambda}}}$ to $e_{\mathbf{\Lambda}}$. It follows that when each of the points $Q_{r}$ equals the associated branch point (i.e., when $t_{r}(Q_{r})=0$ for every such $r$), the numerator on the left hand side of Proposition \ref{quotforZ3} vanishes, while the denominator does not. We omit the exponent 3, and use Equation \eqref{transchar} and the evenness of theta in order to replace the denominator by $\theta[e_{\widetilde{\mathbf{\Lambda}}}](\zeta,\tau)$ and the numerator by $\theta[e_{\mathbf{\Lambda}}](-\zeta,\tau)$ for some small $\zeta$ (the latter vanishing at $\zeta=0$), multiplied by 9th root of unity $\mu$.

Let $q \leq m\leq2q-1$ be the index for which $Q_{m}$ lies near $P_{k}$, and then the third root of the right hand side in Proposition \ref{quotforZ3} becomes
\[\frac{\epsilon_{k}^{1/3}t_{m}(Q_{m})}{f'(\lambda_{k})^{1/3}}\bigg[\prod_{i\in\Lambda_{1}}(\lambda_{k}-\lambda_{i})^{2/3}\prod_{i\in\Lambda_{2}}(\lambda_{k}-\lambda_{i})^{1/3}+\sum_{r=1}^{g}O\big(t_{r}(Q_{r})^{3}\big)\bigg].\] We therefore differentiate both sides with respect to $\zeta_{s}$ at $\zeta=0$, where the left hand side gives us the required derivative, multiplied by $-\mu\big/\theta[e_{\widetilde{\mathbf{\Lambda}}}](0,\tau)$. On the other hand, the factor $t_{m}(Q_{m})$ implies that at $\zeta=0$ (hence $t_{m}(Q_{m})=0$) the derivative in question is simply \[\frac{\epsilon_{k}^{1/3}}{f'(\lambda_{k})^{1/3}}\prod_{i\in\Lambda_{1}}(\lambda_{k}-\lambda_{i})^{2/3}\prod_{i\in\Lambda_{2}}(\lambda_{k}-\lambda_{i})^{1/3}\cdot\frac{\partial t_{m}(Q_{m})}{\partial\zeta_{s}}.\]

For evaluating the last multiplier we invoke Lemma \ref{pardersmult}, in which $r=m$ and $i=k$, the points in question are $P_{i}$ with $i\in\Lambda_{1}\cup\Lambda_{2}\cup\{k\}$, and the polynomial $F_{+}(z)$ is the product of $z-\lambda_{i}$ over the same set of $i$'s. Hence the term denoted by $\sigma_{2q-1-l}^{(r)}(Q_{1},\ldots,Q_{2q-1})$ there is simply $\sigma_{2q-1-l}(\Lambda_{1}\cup\Lambda_{2})$, and collecting the terms of the form $\lambda_{k}-\lambda_{i}$ yields
\[\frac{f'(\lambda_{k})^{1/3}}{3F_{+}'(\lambda_{k})}\prod_{i\in\Lambda_{1}}(\lambda_{k}-\lambda_{i})^{2/3}\prod_{i\in\Lambda_{2}}(\lambda_{k}-\lambda_{i})^{1/3}=
\prod_{i\in\Lambda_{0}\setminus\{k,\infty\}}(\lambda_{k}-\lambda_{i})^{1/3}\Bigg/3\prod_{i\in\Lambda_{2}}(\lambda_{k}-\lambda_{i})^{1/3}.\] We then multiply by the denominator $\theta[e_{\widetilde{\mathbf{\Lambda}}}](0,\tau)$, express it via Theorem \ref{Thomae1Z3}, observe that the $\Delta$-terms from that theorem (with $\widetilde{\mathbf{\Lambda}}$) combine with the latter quotient to the desired $\Delta$-terms (associated with $\mathbf{\Lambda}$) because of the relations between these partitions, and merge all the roots of unity into $\epsilon_{\mathbf{\Lambda}}$. This completes the proof of the theorem.
\end{proof}

\smallskip

The result for the other type of divisors is as follows.
\begin{thm}
Take now $\Xi$ and $\mathbf{\Lambda}$ to be such that
\[|\Lambda_{2}|=q-2\quad\mathrm{and}\quad|\Lambda_{1}|=|\Lambda_{0}|=q+1,\quad\mathrm{but\ now\ with}\quad\infty\in\Lambda_{1},\] with $e_{\mathbf{\Lambda}}$ as in Equation \eqref{charLambda} again. The formula for $1 \leq s \leq g$ is now \[\frac{\partial}{\partial\zeta_{s}}\theta[e_{\mathbf{\Lambda}}](\zeta,\tau)\bigg|_{\zeta=0}=
\Delta(\Lambda_{0})^{1/2}\Delta(\Lambda_{1})^{1/2}\Delta(\Lambda_{2})^{1/2}\Delta(\Lambda_{0},\Lambda_{1})^{1/6}\Delta(\Lambda_{1},\Lambda_{2})^{1/6}\Delta(\Lambda_{2},\Lambda_{0})^{1/6}\times\] \[\times\frac{2\epsilon_{\mathbf{\Lambda}}\alpha}{3}\sqrt{\det C}\sum_{l=2q}^{3q-2}(-1)^{3q-2-l}\sigma_{3q-2-l}(\Lambda_{2})C_{ls},\] with $\alpha$, $C$, and $\epsilon_{\mathbf{\Lambda}}$ as in Theorem \ref{Thomae2Z3t1}. \label{Thomae2Z3t2}
\end{thm}

\begin{proof}
Take again some $k\in\Lambda_{0}$, and define $\Delta$ and $\widetilde{\mathbf{\Lambda}}$ with
\[\widetilde{\Lambda}_{2}=\Lambda_{2}\cup\{k,\infty\},\ \widetilde{\Lambda}_{1}=\Lambda_{0}\setminus\{k\},\mathrm{\ and\ }\widetilde{\Lambda}_{0}=\Lambda_{1}\setminus\{\infty\},\mathrm{\ and\ again\ }|\widetilde{\Lambda}_{2}|=|\widetilde{\Lambda}_{1}|=|\widetilde{\Lambda}_{0}|=q\] (with $\infty\in\widetilde{\Lambda}_{2}$). The points $Q_{r}$, $1 \leq r \leq g$ will again be in the neighborhoods of the branch points appearing in $\Delta$ (with $P_{\infty}$ omitted), with the order from Lemma \ref{pardersmult}, and once again the equality $u_{P_{\infty}}(P_{k})-e_{\widetilde{\mathbf{\Lambda}}}=e_{\mathbf{\Lambda}}$ from Equation \eqref{TPkDelta} holds. Hence with $t_{r}(Q_{r})=0$ for every $r$ the numerator from Proposition \ref{quotforZ3} vanishes and the denominator does not, and we write, using Equation \eqref{transchar}, the third root of that quotient as a 9th root of unity $\mu$ times $\theta[e_{\mathbf{\Lambda}}](-\zeta,\tau)\big/\theta[e_{\widetilde{\mathbf{\Lambda}}}](\zeta,\tau)$, with $\zeta$ close to 0 (the quotient again vanishing at $\zeta=0$).

Since now $P_{k}$ is a double point of $\Delta$, we define $1 \leq m\leq q-1$ to such that $Q_{m}$ and $Q_{m+2q-1}$ both lie near $P_{k}$, and the right hand side in Proposition \ref{quotforZ3} becomes (after taking the third root), in the coordinates $t_{r}(Q_{r})$, \[\frac{\epsilon_{k}^{1/3}t_{m}(Q_{m})^{2}}{f'(\lambda_{k})^{1/3}}\bigg[\prod_{i\in\Lambda_{2}}(\lambda_{k}-\lambda_{i})^{2/3}\prod_{i\in\Lambda_{0}\setminus\{k\}}(\lambda_{k}-\lambda_{i})^{1/3}
+\sum_{r=1}^{g}O\big(t_{r}(Q_{r})^{3}\big)\bigg].\] The derivative of the left hand side with respect to $\zeta_{s}$ at $\zeta=0$ yields the required derivative times $-\mu\big/\theta[e_{\widetilde{\mathbf{\Lambda}}}](0,\tau)$ again, while the existence of the factor $t_{m}(Q_{m})^{2}$ again shows that only an expression involving the second derivative with respect to $t_{m}(Q_{m})$ survives on the right hand side at $\zeta=0$. More precisely, applying Lemma \ref{symfuncder} and using Equation \eqref{coordinates} shows that the derivative on the left hand side at $\zeta=0$ reduces to \[\frac{\epsilon_{k}^{1/3}}{f'(\lambda_{k})^{1/3}}\prod_{i\in\Lambda_{2}}(\lambda_{k}-\lambda_{i})^{2/3}\prod_{i\in\Lambda_{0}\setminus\{k\}}(\lambda_{k}-\lambda_{i})^{1/3}\cdot\frac{\partial\beta_{m}}{\partial\zeta_{s}}\] (the factor 2 from the denominator in Lemma \ref{symfuncder} cancels with the fact that $\frac{\mathrm{d}^{2}}{\mathrm{d}t^{2}}t^{2}=2$).

The last multiplier is again evaluated by taking $r=m$ and $i=k$ in Lemma \ref{pardersmult}, where now the $q-1$ points are $P_{i}$ with $i\in\Lambda_{2}\cup\{k\}$, and $F_{-}(z)$ is $\prod_{i\in\Lambda_{2}\cup\{k\}}(z-\lambda_{i})$. Therefore $\sigma_{3q-2-l}^{(r)}(Q_{1},\ldots,Q_{q-1})$ is just $\sigma_{3q-2-l}(\Lambda_{2})$, and the powers of the terms $\lambda_{k}-\lambda_{i}$ combine to \[\frac{2}{3F_{-}'(\lambda_{k})}\prod_{i\in\Lambda_{2}}(\lambda_{k}-\lambda_{i})^{2/3}\prod_{i\in\Lambda_{0}\setminus\{k\}}(\lambda_{k}-\lambda_{i})^{1/3}=
2\prod_{i\in\Lambda_{0}\setminus\{k\}}(\lambda_{k}-\lambda_{i})^{1/3}\Bigg/3\prod_{i\in\Lambda_{2}}(\lambda_{k}-\lambda_{i})^{1/3}.\] Once again the $\Delta$-terms from the expression for the denominator $\theta[e_{\widetilde{\mathbf{\Lambda}}}](0,\tau)$ in Theorem \ref{Thomae1Z3} and the latter quotient yield the required $\Delta$-terms (for $\mathbf{\Lambda}$), the constants are as stated, and we denote the remaining root of unity by $\epsilon_{\mathbf{\Lambda}}$. This completes the proof of the theorem.
\end{proof}

\begin{cor}
The formula from Theorem \ref{Thomae2Z3t2} is valid also when $\infty\in\Lambda_{0}$. \label{inftyLambda0}
\end{cor}

\begin{proof}
It is easy to see that when interchanging $\Lambda_{0}$ and $\Lambda_{1}$ takes the characteristic $e_{\mathbf{\Lambda}}$ to $-e_{\mathbf{\Lambda}}$. The result thus follows from Theorem \ref{Thomae2Z3t2} and the evenness of theta functions with respect to inverting characteristics, since all the expressions in the formula are invariant under this interchange (up to modifying the root of unity $\epsilon_{\mathbf{\Lambda}}$). This proves the corollary.
\end{proof}
An inversion of the characteristics, as in the proof of Corollary \ref{inftyLambda0}, can be applied equally well to the formula from Theorem \ref{Thomae2Z3t1}. However, in this case we interchange $\Lambda_{1}$ with $\Lambda_{2}$, and the resulting divisors is already covered in Theorem \ref{Thomae2Z3t1} itself. In particular we see an interesting phenomenon: The upper $2q-1$ rows of the matrix $C$ appear in derivative formulae for divisors (or partitions) of the first type via Theorem \ref{Thomae2Z3t1}, while the lower $q-1$ rows of that matrix show up in the formulae associated with divisors (or partitions) of the second type, as in Theorem \ref{Thomae2Z3t2} and Corollary \ref{inftyLambda0}.

\smallskip

Once again we would like to express the formulae arising from the divisors that are ``near'' a non-special divisor $\Delta$ of degree $g$ in matrix form, as in Equation \eqref{matrixthomae2}. Given $\Delta$ and $\mathbf{\Lambda}$ as in Theorem \ref{Thomae1Z3}, the possible ways to obtain from $\Delta$ positive divisors as in Theorem \ref{divsforder} while affecting only one branch point (and $P_{\infty}$) is to subtract $P_{i}-P_{\infty}$ from $\Delta$ for $i\in\Lambda_{1}\cup\Lambda_{2}\setminus\{\infty\}$, or subtract $2P_{i}-2P_{\infty}$ for any $i\in\Lambda_{2}\setminus\{\infty\}$. Let us organize the index $k$ such that for $1 \leq k \leq q$ we define $\mathbf{\Lambda}^{(k)}$ to arise from an index $i\in\Lambda_{1}$, then when $q+1 \leq k\leq2q-1$ we consider $i\in\Lambda_{2}\setminus\{\infty\}$ with a simple subtraction, and if $2q \leq k\leq3q-2$ then we take $i\in\Lambda_{2}\setminus\{\infty\}$ but with a double subtraction. More explicitly we get \[\left\{\begin{array}{cccc} \Lambda_{2}^{(k)}=\Lambda_{2}\setminus\{\infty\} & \Lambda_{1}^{(k)}=\Lambda_{1}\setminus\{i\} & \Lambda_{0}^{(k)}=\Lambda_{0}\cup\{i,\infty\} & 1 \leq k \leq q \\ \Lambda_{2}^{(k)}=\Lambda_{2}\setminus\{i,\infty\} & \Lambda_{1}^{(k)}=\Lambda_{1}\cup\{i\} & \Lambda_{0}^{(k)}=\Lambda_{0}\cup\{\infty\} & q+1 \leq k\leq2q-1 \\ \Lambda_{2}^{(k)}=\Lambda_{2}\setminus\{i,\infty\} & \Lambda_{1}^{(k)}=\Lambda_{1}\cup\{\infty\} & \Lambda_{0}^{(k)}=\Lambda_{0}\cup\{i\} & 2q \leq k\leq3q-2=g.\end{array}\right.\] Then we set $\mathcal{D}$ to be the diagonal matrix with $kk$-entry \[\epsilon_{\mathbf{\Lambda}^{(k)}}\Delta(\Lambda_{0}^{(k)})^{1/2}\Delta(\Lambda_{1}^{(k)})^{1/2}\Delta(\Lambda_{2}^{(k)})^{1/2}
\Delta(\Lambda_{0}^{(k)},\Lambda_{1}^{(k)})^{1/6}\Delta(\Lambda_{1}^{(k)},\Lambda_{2}^{(k)})^{1/6}\Delta(\Lambda_{2}^{(k)},\Lambda_{0}^{(k)})^{1/6},\] and we define $\Sigma$ via \[\Sigma_{kl}=\left\{\begin{array}{cc} (-1)^{2q-1-l}\sigma_{2q-1-l}(\Lambda_{1}^{(k)}\cup\Lambda_{2}^{(k)}) & 1 \leq k \leq q\mathrm{\ and\ }1 \leq l\leq2q-1 \\ 2(-1)^{3q-2-l}\sigma_{3q-2-l}(\Lambda_{2}^{(k)}) & q+1 \leq k\leq3q-2\mathrm{\ and\ }2q \leq l\leq3q-2 \\ 0 &  \mathrm{otherwise}.\end{array}\right.\] With this notation we find, in analogy with Equation \eqref{matrixthomae2}, that
\begin{equation}
\frac{\partial\big(\theta[e_{\mathbf{\Lambda}^{(1)}}],\ldots,\theta[e_{\mathbf{\Lambda}^{(g)}}]\big)}{\partial(\zeta_{1},\ldots,\zeta_{g})}
\bigg|_{\zeta=0}=\frac{\alpha}{3}\cdot\mathcal{D}\Sigma C. \label{degenmat}
\end{equation}

As for Theorem \ref{Thomae2Z2}, the results of Theorems \ref{Thomae2Z3t1} and \ref{Thomae2Z3t2} are equally valid without the assumption on the location of $\infty$ in $\mathbf{\Lambda}$, with the $\Delta$ terms remaining unaffected, but where every expression of the sort $\sigma_{p}(\Lambda)$ (with $\Lambda$ being either $\Lambda_{1}\cup\Lambda_{2}$ in Theorem \ref{Thomae2Z3t1}, or just $\Lambda_{2}$ as in Theorem \ref{Thomae2Z3t2}) should be replaced by $\sigma_{p-1}(\Lambda\setminus\{\infty\})$. This is proved by taking $k\in\Lambda_{0}$ (for both types of partitions), comparing $\mathbf{\Lambda}$ with the partition $\widetilde{\mathbf{\Lambda}}$ in which $k$ and $\infty$ are interchanged, and noting that an appropriate choice of points $\{Q_{r}\}_{r=1}^{g}$ in Proposition \ref{quotforZ3} would give us (up to some extra root of unity) the quotient $\frac{\theta[e_{\mathbf{\Lambda}}](\zeta,\tau)}{\theta[e_{\widetilde{\mathbf{\Lambda}}}](\zeta,\tau)}$ with some small $\zeta$, which we take to be $\delta$ times the $s$-th standard vector, with $\delta$ small. Expanding everything in terms of $\delta$, and noticing that the two theta constants vanish (at $\delta=0$), we get a explicit expression for a quotient involving our desired theta derivative and a theta derivative that is already evaluated via Theorem \ref{Thomae2Z3t1} or \ref{Thomae2Z3t2}. An analogue of Corollary \ref{inftyLambda0} then completes the proof for a few remaining partitions.

As a final remark we state that when the polynomial $f$ from Equation \eqref{Z3curve} has degree $3q$ (instead of $3q-1$), and we choose the base point $P_{0}$ to be with a finite $z$-value $\lambda_{0}$, the results of Theorems \ref{Thomae2Z3t1} and \ref{Thomae2Z3t2} etc. continue to hold, but with the proofs involving a few additional terms.

\bigskip

On behalf of all authors, the corresponding author states that there is no conflict of interest.

\noindent\textsc{National University of Kyiv-Mohyla Academy, H. Skovorody St. 2, 04655 Kyiv, Ukraine \\ Finance Department School of Business 2100, Hillside University of Connecticut, Storrs, CT 06268 \\ Einstein Institute of Mathematics, the Hebrew University of Jerusalem, Edmund Safra Campus, Jerusalem 91904, Israel}

\noindent E-mail: venolski@googlemail.com, yaacov.kopeliovich@uconn.edu, zemels@math.huji.ac.il


\begin{thebibliography}{}{}

\bibitem[Ba]{[Ba]} Baker, H. F. \textsc{Abel's Theorem and the Allied Theory of Theta Functions}, Cambridge Univ. Press, Cambridge (1897). Reprinted in 1995.
\bibitem[Be]{[Be]} Bernatska, J. \textsc{General Derivative Thomae Formula for Singular Half-Periods}, pre-print., https://arxiv.org/abs/1904.09333.
\bibitem[BR1]{[BR1]} Bershadsky, M., Radul, A., \textsc{Conformal Field Theories with Additional $Z_{N}$ Symmetry} Int. J. Mod. Phys. A, vol 2, 165--178 (1987).
\bibitem[BR2]{[BR2]} Bershadsky, M., Radul, A., \textsc{Fermionic Fields on $Z_{n}$ curves}, Comm. Math. Phys., vol 116, 689--700 (1988).
\bibitem[D]{[D]} Dubrovin, B., \textsc{Theta-Functions and Nonlinear Equations}, Uspekhi Matem. Nauk, vol 36, 11--92 (1981)
\bibitem[EbF]{[EbF]} Ebin, D., Farkas, H. M., \textsc{Thomae Formulae for $Z_{n}$ Curves}, J. Anal. Math., vol 111, 289--320 (2010).
\bibitem[EMO]{[EMO]} Eilbeck, J. C., Matsutani, S., \^{O}nishi, Y., \textsc{Addition Formulae for Abelian Functions Associated with Specialized Curves}, Phil. Trans. R. Soc. A, vol 369, 1245--1263 (2011).
\bibitem[E]{[E]} Eilers, K., \textsc{Rosenhain-Thomae Formulae for Higher Genera Hyperelliptic Curves}, J. Nonlinear Math .Phys., vol 25 issue 1, 85--105 (2018).
\bibitem[EiF]{[EiF]} Eisenmann, A., Farkas, H. M., \textsc{An Elementary Proof of Thomae's Formulae}, OJAC, vol 3 (2008).
\bibitem[EG]{[EG]} Enolskii, V. Z., Grava, T., \textsc{Thomae Type Formulae for Singular $Z_{N}$ Curves}, Lett. Math. Phys., vol 76 no. 2-3, 187--221 (2006).
\bibitem[ER]{[ER]} Enolskii, V. Z., Richter, P. \textsc{Periods of Hyperelliptic Integrals Expressed in Terms of $\theta$-Constants by Means of Thomae formulae}, Philos. Lond. Trans. R. Soc. Ser. A Math. Phys. Eng. Sci., vol 366 no. 1867, 1005--1024 (2008) .
\bibitem[F1]{[F1]} Fay, J. D., \textsc{Theta Functions on Riemann Surfaces}, Lecture Notes in Mathematics 352, Springer--Verlag, v+133pp (1973).
\bibitem[F2]{[F2]} Fay, J. D., \textsc{On the Riemann–-Jacobi Formula}, Nachrichten der Akadedemie der Wissenschaften in G\"{o}ttingen. II. Mathematisch-Physikalische Klasse 5, 61–73 (1979).
\bibitem[FK]{[FK]} Farkas, H. M., Kra, I., \textsc{Riemann Surfaces}, Graduate Text in Mathematics 71, Springer--Verlag, 354pp (1980).
\bibitem[FZ]{[FZ]} Farkas, H. M., Zemel, S., \textsc{Generalizations of Thomae's Formula for $Z_{n}$ curves}, DEVM 21, Springer--Verlag, xi+354pp (2011).
\bibitem[Ko]{[Ko]} Kopeliovich, Y., \textsc{Thomae Formula for General Cyclic Covers of $\mathbb{CP}^{1}$}, Lett. Math. Phys., vol 94 issue 3, 313--333 (2010).
\bibitem[KZ1]{[KZ1]} Kopeliovich, Y., Zemel, S., \textsc{On Spaces Associated with Invariant Divisors on Galois Covers of Riemann Surfaces and Their Applications}, to appear in Isr. J. Math., https://arxiv.org/abs/1609.02296
\bibitem[KZ2]{[KZ2]} Kopeliovich, Y., Zemel, S. \textsc{Thomae Formula for Abelian Covers of $\mathbb{CP}^{1}$}, to appear in Trans. Amer. Math. Soc., https://arxiv.org/abs/1612.09104.
\bibitem[MT]{[MT]} Matsumoto, K., Tomohide, T., \textsc{Degenerations of Triple Covering and Thomae's Formula}, https://arxiv.org/abs/1001.4950.
\bibitem[N]{[N]} Nakayashiki, A., \textsc{On the Thomae Formula for $Z_{N}$ Curves}, Publ. Res. Inst. Math Sci., vol. 33 issue 6, 987--1015 (1997).
\bibitem[R]{[R]} Rosenhain, G., \textsc{Abhandlung \"{u}ber die Funktionen zweier Variablen mit vier Perioden}, \em M\'{e}m. pr\'{e}s. l'Acad. de Sci. de France des savants, XI:361--455 (1851). The paper is dated 1846. German Translation: H. Weber (Ed.), Engelmann-Verlag, Leipzig (1895).
\bibitem[T1]{[T1]} Thomae, C. J., \textsc{Bestimmung von $\mathrm{d}\log\theta(0,\ldots,0)$ durch die Klassmoduln}, J. Reine Angew. Math., vol 66, 92--96 (1866).
\bibitem[T2]{[T2]} Thomae, C. J., \textsc{Beitrag zur Bestimmung von $\theta(0,\ldots,0)$ durch die Klassmoduln Algebraischer Funktionen}, J. Reine Angew. Math., vol 71, 201--222 (1870).
\bibitem[Z]{[Z]} Zemel, S. \textsc{Thomae Formulae for General Fully Ramified $Z_{n}$ Curves}, J. Anal. Math., vol 131, 101--158 (2017).

\end{thebibliography}
\end{document}